\documentclass{amsart}
\usepackage{amsmath,graphicx}

\usepackage{amsthm}
\usepackage{amssymb}
\usepackage{xypic}
\usepackage[square]{natbib}
\usepackage{algorithm,algorithmic}
\usepackage{caption}
\usepackage{hyperref}
\setcitestyle{numbers}
\theoremstyle{plain}
\newtheorem{theorem}{Theorem}
\newtheorem{proposition}{Proposition}
\newtheorem{corollary}{Corollary}
\newtheorem{lemma}{Lemma}
\theoremstyle{definition}
\newtheorem{definition}{Definition}

\newtheorem{remark}{Remark}
\newtheorem{example}{Example}

\def\cat{\mathbf}

\usepackage{xcolor}
\newcommand{\N}{\mathbb{N}}
\newcommand{\rel}{\cat{Rel}}
\newcommand{\relX}{\cat{Rel}_{X}}
\newcommand{\mul}{\cat{Mul}}
\newcommand{\mulX}{\cat{Mul}_{X}}
\newcommand{\pis}{\pi_{*}}
\newcommand{\hgt}{\text{hgt}}

\title{Metric Comparisons of Relations}
\author{Kenneth P. Ewing}
\address{4201 Cathedral Ave, NW Apt 715E Washington, DC 20016}
\email{kenneth.p.ewing@gmail.com}
\author{Michael Robinson}
\address{Mathematics and Statistics\\
American University\\
Washington, DC, USA}
\email{michaelr@american.edu}

\begin{document}
\begin{abstract}
  This paper defines a new pseudometric for binary relations between finite sets that measures consensus among subsets. The main results are (1) a concise restatement of this pseudometric with an intuitively appealing interpretation via a full and faithful functor from the category of relations to a category of relation multisets and (2) that the pseudometric can be bounded without an expensive search of possible mappings, based solely on the dimensions of the relations themselves. Additionally, (3) an algorithm is described to calculate this bound with time and memory complexity at worst proportional to the product of those dimensions: $\mathcal{O}(m \times n)$.  The tools developed in this paper should find broad application in exploratory data analysis.  We provide one such application by briefly exploring \emph{ad hoc} consensus specifications for the well-known PDF file format.
\end{abstract}

\maketitle

\tableofcontents

\section{Introduction}

We classify people, things, ideas all the time. At the heart of classification is some rule for discriminating: ``$\alpha$ is in group $X$, because $\alpha$ exhibits features $\{a, c, q\}$.'' It can be hard enough to measure and handle uncertainties about the features of $\alpha$, and so on; but what if we are unsure what the discrimination rule even is? What is moral and what is immoral? One common solution is to look at classifications made by others we trust to find a consensus rule, some specification that is perhaps not universal but axiomatically ``enough'' to act with confidence. Mathematical tools to help find consensus rules like that abound. For example, principal component analysis and factor analysis use the covariance of feature measurements to identify principal or unobserved variables that independently accumulate variance in data (see, e.g., \cite[]{dunteman_principal_1989}, \cite[]{hanafi_connections_2011}). Formal concept analysis uses the partial ordering of common feature sets in incidence matrices to construct a hierarchy of formal concepts that organize and cover the data (see, e.g., \cite[]{ganter_formal_1999}, \cite[]{skopljanac-macina_formal_2014}, \cite[]{dias_concept_2015}). Topological data analysis uses persistent homology and other techniques to identify stable or invariant characteristics of a space sampled by data (see, e.g., \cite[]{ghrist_elementary_2014}, \cite[]{curry_topological_2015}). All can be used to discover or construct a set of possibly latent characteristics and relationships among them that can describe all or most of the data (see, e.g., \cite[]{janostik_interface_2020}). Applied to data from classifications by trusted agents, these tools can give us a consensus rule to guide our own decisions.

A concrete example, which motivated this paper, arises from the humble ``PDF'' document. The Portable Document Format, originally developed by Adobe Systems in the 1990s and released for standardization by the International Organization for Standardization as ISO 32000-1:2008 \cite[]{iso_pdf_2008},\footnote{The standard was recently updated to ISO 32000-2:2021. (\href{https://www.iso.org/standard/75839.html}{https:\slash\slash www.iso.org\slash standard\slash75839.html}).} specifies how to represent electronic documents so that they can be displayed on screen or paper uniformly, independent of the software or hardware environment. Each PDF encodes what and where to place marks on a page in part using a subset of the PostScript page description programming language. The scheme works so well that PDFs have become a format of choice for replacing paper with electronic versions.
PDF parsers are required by the standard to support powerful capabilities like the ability to encapsulate data and to calculate transformations to accommodate different reading environments.
However, these capabilities and the wide adoption of the standard make PDFs attractive vectors for electronic attacks on computer systems.
Indeed, malicious payloads have been demonstrated since at least 2001 (see, e.g., \cite[]{wikipedia_pdf_nodate} (citing announcement on Adobe Systems online forum August 15, 2001); \cite[]{muller_processing_2021} (discussing attacks using compliant PostScript)). This has led to research into fast and extremely scalable techniques to classify PDFs as ``safe''or to quarantine potentially ``unsafe'' ones (see \cite[]{bratus_safe_nodate}). Unfortunately, the PDF standard is so complex that ensuring complete compliance with it has not been possible, at least to date. Consequently, PDF readers and parsers do not always render the same PDF identically, and, more ominously, they do not always agree on whether or why to reject a PDF. In short, we do not have a universal specification of a ``safe'' PDF. We need to find a consensus rule.

In recent research our team has developed the concept of a \emph{weighted Dowker complex} to represent binary relations in order to use topological tools to reveal latent consensus rules. For instance, in the context of classifying PDFs as safe, the data can be binary decisions---``error'' v. no ``error''---about each of several hundred thousand sample PDFs by a dozen different readers or parsers \cite[]{robinson_2021_looking}. This amounts to a relation among documents and parsers, for which a weighted Dowker complex can be constructed with parsers as the vertex set \cite[]{Ambrose_2020}.  We recently demonstrated that a cosheaf of abstract simplicial complexes constructed from the weighted Dowker complex faithfully represents the relation \cite[]{robinson_cosheaf_2020}.

An ideal consensus, of course, should be stable in the sense that relatively small changes in the relation should not destroy the consensus. To help assess the degree of consensus we define a pseudometric distance for relations akin to a Hamming distance (definition \ref{df:distance}).
Like a Hamming distance, this relation distance is a function of potential mappings between spaces represented by the relations. We explore this pseudometric in an attempt to find practical ways to calculate or estimate it without requiring an exhaustive search of mappings, particularly in light of the potential application to large relations. The paper explores limitations of the Dowker complex representation of a relation and then uses a variation on Monro's seminal categorial definition of multisets \cite[]{monro_concept_1987} to analyze our relation distance. Petrovsky has outlined a range of other metrics of similarity or dissimilarity among multisets \cite[]{petrovsky_metrics_2019}. The focus here, however, is on the distance metric we have defined. Viewing relations as boolean matrices, it has been shown that decomposing them into a set of basis matrices---which can be interpreted as representing a latent consensus---is an NP-hard problem (see \cite[]{trnecka_data_2018}, fn. 1 (citing \cite{Stockmeyer_1975})) and computing formal concept lattices can be extremely expensive because of their combinatorial complexity \cite[]{dias_methodology_2017}. The aim here is not to compute the latent consensus itself but to help make our distance metric practical in the search for topological methods to identify a latent consensus.

The main results are that
\begin{enumerate}
  \def\labelenumi{(\arabic{enumi})}
  \item a natural metric can be defined on the space of relations (definition \ref{df:distance}), akin to a Hamming distance,
  \item a full and faithful functor from the category of relations to a category of relation multisets enables a concise restatement of the pseudometric with an intuitively appealing interpretation (theorem \ref{theorem:mulRfunctor});
  \item relation distance can be bounded based solely on the dimensions of the relations themselves, without expensively searching the space of possible mappings (theorem \ref{theorem:bounds}); and
  \item the $\kappa$ algorithm calculates this bound with time and memory complexity proportional to the product of those dimensions: $\mathcal{O}(m \times n)$ (section \ref{section:algorithm}). A \texttt{python3} script implementing the algorithm may be found at \href{https://github.com/kpewing/relations.git}{https:\slash\slash github.com\slash kpewing\slash relations.git}.
\end{enumerate}

The plan of the paper is as follows. After a brief motivation in Section \ref{sec:motivation}, we define our category of relations and explore some examples to tease out some of the complexities they embody in Section \ref{sec:relations}. We then consider Dowker complexes and their limitations for analyzing distance between relations in Section \ref{sec:dowker}. With these preliminaries in hand, we define a category of relation multisets and prove the existence of a full and faithful functor to it from the category of relations in Section \ref{sec:multiset_functor}. The heart of the paper is Section \ref{sec:calculating}, which uses the multiset representation to restate our distance in terms that enable both an intuitive interpretation and bounds that can be calculated practically using the algorithm. In Section \ref{sec:conclusion}, we conclude with some general observations and suggestions for further research.

\section{Motivation}
\label{sec:motivation}

Consider the PDF file classification problem mentioned in the introduction.  Suppose we would like to disposition a set of files $X$ as ``good'' or ``bad'' given two additional sets of files for comparison: one set $S$ of ``safe'' files and one set $U$ of ``unsafe'' files.  Given a pseudometric $d$ for sets of files, this problem reduces to that of comparing $d(X,S)$ with $d(X,U)$.

The methodology proposed in this paper is that we regard not just a set $X$ of files, but rather a relation between a set of files and a set of messages produced by a collection of parsers.  That is, for each file and for each potential message, we record whether that message was produced when attempting to read the file.  Note that a message could be, but need not be, an \emph{error} message that indicates a problem with the file.  Parsers also often produce \emph{informational} messages that do not indicate a judgement about the validity of a file.  We can store this information in a tabular matrix format, with a row for each message and a column for each file.  Such a matrix is a convenient representation of the relation between files and messages.  Patterns in the presence or absence of messages can reveal subtle information about the safety of a given file.  For instance, some files are good but produce harmless error messages from particularly stringent parsers, while dangerous files may evade detection when read by more lenient parsers.

\begin{figure}
  \begin{center}
    \includegraphics[width=5in]{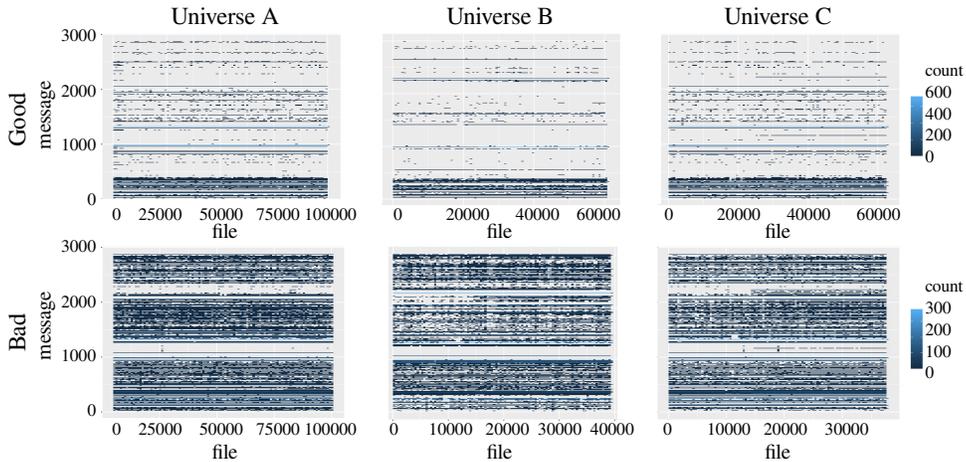}
    \caption{Matrices for the six relations discussed in the text.  The meanings of the rows are the same across all six matrices, and correspond to the set of $3023$ standardized regular expressions matched against file output.  The columns correspond to each file present in the corresponding set, and differ between matrices.  Dark colors correspond to the number of times a corresponding regular expression was found in the output for a given file.  Gray corresponds to the absence of a match for a particular regular expression on a given file.}
    \label{fig:relation_matrices}
  \end{center}
\end{figure}

In a recent processing campaign, which was a follow-on effort to the one discussed in \cite{robinson_2021_looking}, $400004$ files were provided to us by the DARPA SafeDocs Test and Evaluation Team.  The Test and Evaluation Team divided the files into three ``universes'' based on expected semantic and syntactic features, and further divided these universes into ``good'' and ``bad'' files, as follows:
\begin{itemize}
  \item Universe A: $100001$ good files, $100001$ bad files,
  \item Universe B: $60737$ good files, $39264$ bad files, and
  \item Universe C: $62408$ good files, $37593$ bad files.
\end{itemize}
The Test and Evaluation Team ensured that the Universe A bad files had problems of a rather general nature, while the other two sets of bad files were of several specific varieties, many of which did not produce error messages upon parsing.

We read the $40004$ PDF files using a set of $28$ parsers.  Rather than recording one row per parser, which would seem rather coarse, we classified the error messages produced by the parsers by using a standardized set of $3023$ regular expressions run against the output of each file.  This process is easy to deploy uniformly across all files and parsers.  These regular expressions became the rows of our relation's matrix representation.  Conversely, each file corresponds to a column in the relation's matrix representation.

Figure \ref{fig:relation_matrices} shows the resulting six relation matrices: three for files that are good, three for files that are bad.  Visually, the good and bad files' matrices are quite distinct.  The good files generally tend to produce fewer messages that match our regular expressions, which explains the presence of gray in their corresponding frames in Figure \ref{fig:relation_matrices}.  (Indeed, the good files produced far fewer messages on average than the bad files, not just fewer matches against our regular expression set.)  Notice that all of the six matrices share the same row definitions, since these correspond to the regular expressions, but the columns cannot be directly compared between relations.  Managing the incompatibility of columns forms the bulk of the effort in the latter sections of this paper.

\begin{figure}
  \begin{center}
    \includegraphics[height=3in]{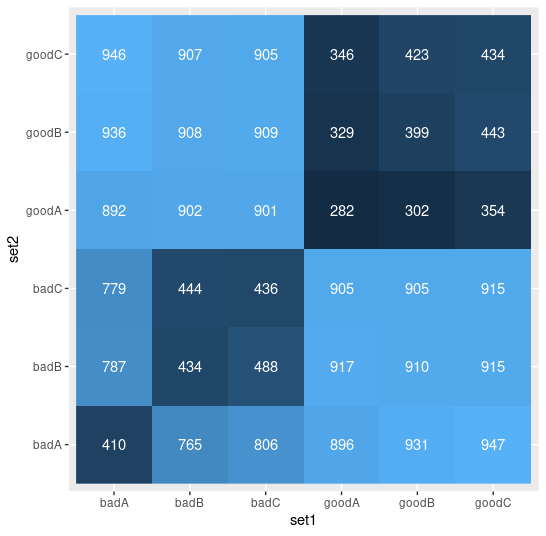}
    \caption{The upper bound (see Theorem \ref{theorem:bounds} in Section \ref{sec:general_statement}) on relation distance between each pair of relations in Figure \ref{fig:relation_matrices}.  The bound was computed using $1000$ files drawn from each relation in each pair. Lighter cell colors reflect higher relation distance.}
    \label{fig:relation_distances}
  \end{center}
\end{figure}

Using the upper bound presented in Theorem \ref{theorem:bounds} in Section \ref{sec:general_statement} of this paper, we can estimate the distance between each pair of relations in Figure \ref{fig:relation_matrices}.  The results of this calculation are shown in Figure \ref{fig:relation_distances}.  Before interpreting the results, a few cautions are worthy of note.  Although the matrix of pairwise distances should ideally have zero diagonal, this is not the case for an upper bound on the distance.  The bound on the distance from a relation to itself could be (and is) larger than zero in our case.  Typically the diagonal is the smallest distance in each row of Figure \ref{fig:relation_distances} with the exception of the first two rows, where the diagonal is still quite close to the smallest value.  Secondly, the matrix of pairwise distances ought to be symmetric across the diagonal, even though Figure \ref{fig:relation_distances} exhibits a slight asymmetry.  This is entirely due to sampling error.  For efficiency reasons, each pairwise distance was computed independently with a random draw of $1000$ files (columns) from each relation being compared.  These draws were redone for each pair of relations, which results in a small change in distance whenever the calculation is repeated.  Finally, the values of the distances exhibited in Figure \ref{fig:relation_distances} reflect the maximum of $1000$, since this is the number of columns drawn for the estimate.

Aside from these cautions, the large values (corresponding to highest relation distance and lightest colors) in the off-diagonal blocks of Figure \ref{fig:relation_distances} justify our intuition that the good files and bad files are quite different.  Furthermore, it is evident that the good files are all quite similar to each other, having the lowest relation distances and darkest cells in the upper right quadrant.  The bad files are much more variable, with the bad files in Universe A being rather distinct from the others, as reflected in the separation of lighter and darker cells within the lower left quadrant.  Indeed, this separation is precisely what we would expect since Universe A bad files had problems that generated error messages, while the other two Universes' bad files did so less frequently.

The remainder of the paper provides all necessary details to compute the relation pseudometric and its upper bound.  Additionally, it proves several theoretical guarantees about the behavior of the pseudometric.

\section{Relations}
\label{sec:relations}

Let $R \subseteq X \times Y$ be a relation between finite sets $X$ and $Y$, which can be represented as a Boolean matrix $(r_{x,y})$.

\begin{definition}
  \label{df:rel_cat} (which has \cite[Sec 3.3]{JoyofCats} or \cite[pg. 54]{Rydeheard_1988} as a special case, and is manifestly the same as what appears in \cite{brun2019sparse})
  The \emph{category of relations} ${\bf Rel}$ has triples $(X,Y,R)$ for objects.  A morphism $(X,Y,R) \to (X',Y',R')$ in ${\bf Rel}$ is defined by a pair of functions $f:X\to X'$, $g:Y\to Y'$ such that $(f(x),g(y))\in  R'$ whenever $(x,y)\in R$.  Composition of morphisms is simply the composition of the corresponding pairs of functions, which means that ${\bf Rel}$ satisfies the axioms for a category.
\end{definition}

  It will be useful to consider the full subcategory ${\bf Rel_+}$ of ${\bf Rel}$ in which each object $(X,Y,R)$ has the property that for each $x \in X$, there is a $y \in Y$ such that $(x,y) \in R$, and conversely for each $y \in Y$, there is an $x \in X$ such that $(x,y) \in R$. For objects in this subcategory the matrix representation of $R$ has no zero rows or columns.

The situation we want to consider is that of two relations $R_1 \subseteq (X \times Y_1)$ and $R_2 \subseteq (X \times Y_2)$, sharing the set $X$.  (In the case of the PDF file analysis in Section \ref{sec:motivation}, $X$ consists of the set of message regular expressions.)  This suggests that we restrict our attention from the category ${\bf Rel}$ of all relations to the category ${\bf Rel}_X$ of relations on a fixed set $X$. The subcategory ${\bf Rel}_X$ has pairs $(Y,R)$ for objects, in which $R \subseteq (X \times Y)$.  Abusing notation slightly, we will think of each object $(Y,R)$ of ${\bf Rel}_X$ as an object $(X,Y,R)$ of ${\bf Rel}$.  Each morphism of ${\bf Rel}_X$ is a morphism $(id_X,g)$ of ${\bf Rel}$ in which the first map is an identity map on $X$.  We will usually simply write $g: (Y,R) \to (Y',R')$ for a morphism of ${\bf Rel}_X$.  Specifically a function $g: Y \to Y'$ will be considered a morphism of ${\bf Rel}_X$ if $(x,g(y)) \in R'$ whenever $(x,y) \in R$.

Notice that ${\bf Rel}_X$ is not a full subcategory of ${\bf Rel}$, since we do not allow a morphism of ${\bf Rel}_X$ to transform $X$ even though a morphism of ${\bf Rel}$ could indeed do that even if two relations share $X$ as their first set.

We can consider comparing two objects $(Y,R)$ and $(Y',R')$ in $\relX$ by using something like an edit distance, but permitting elements of $Y$ to be permuted. That is, we can use a measure of changes induced by any function mapping $Y$ components of the two objects,
called \emph{weight}, to define a \emph{distance} between $\relX$ objects that is a pseudometric.

\begin{definition}
  \label{df:weight}
Given two $\relX$ objects $r_1 = (Y_1, R_1)$ and $r_2=(Y_2,R_2)$ and a function $g: Y_1 \to Y_2$, we will usually write $g: r_1 \to r_2$, and will define the \emph{weight} of $g$ as
  \begin{equation*}
    w(g|r_1, r_2) := \max_{x\in X} \left\{ \#Y_2\!\setminus\! g(Y_1) +\! \sum_{y\in Y_1}\! \left\{\begin{array}{l}0\ \text{if}\ ( (x,y)\in R_1\ \text{and}\ (x, g(y))\in R_2))\\\phantom{0}\ \text{or}\ ((x,y)\notin R_1\ \text{and}\ (x,g(y)\notin R_2))\\1\ \text{otherwise}\end{array} \!\!\right\} \!\right\}
  \end{equation*}
\end{definition}

\begin{definition}
  \label{df:distance} The \emph{distance} between two $\relX$ objects $r_1$ and $r_2$ is defined as \begin{equation*}
    d(r_1,r_2) := \max\ \{\min_{g: r_1\to r_2}\negthickspace w(g|r_1, r_2),\ \min_{g': r_2\to r_1}\negthickspace w(g|r_1, r_2)\ \}
  \end{equation*} where the $g$ and $g'$ are allowed to range over all functions, without restriction.
\end{definition}

Note that these definitions apply not only to $\relX$ morphisms but also to functions that are not $\relX$ morphisms. As will be discussed shortly (see Example \ref{eg:eg_nomorph_emptyrel}) and explored more deeply later (see Proposition \ref{proposition:morphismweight}), this general applicability is important, because morphisms do not always exist between two arbitrary $\relX$ objects. Consequently, we will use ``morphism'' only when referring to a category morphism and will use ``function'' or ``mapping'' when not restricted to morphisms, \emph{even when used with respect to category objects.}

The weight combines for each $x\in X$ the number of times $x$'s relations change under the mapping $g:Y_1\to Y_2$ and the number of wholly ``new'' relations for $x$ in the target $r_2$ (since outside the range of $g$), and then yields the largest of these sums. It can thus be understood as measuring the largest change induced by $g$ for any $x$; and the smallest weight over all functions $g\colon r_1\to r_2$ can be understood as counting edits from $r_1\to r_2$. Taking the larger of the edit counts in both directions, i.e., from $r_1\to r_2$ and from $r_2\to r_1$, is what defines the distance between the two relations; this is indeed a \emph{pseudometric}.

\begin{proposition}
  \label{prop:disapseudometric} The distance $d$ is a pseudometric for $\relX$, i.e., it satisfies for all $\relX$ objects $r_1, r_2, r_3$:
  \begin{description}
  \item[\textbf{Nonnegativity}] $d(r_1, r_2) \ge 0$
  \item[\textbf{Symmetry}] $d(r_1, r_2) = d(r_2, r_1)$
  \item[\textbf{Reflexivity}] $d(r_1, r_1) = 0$
  \item[\textbf{Triangle inequality}]
    $d(r_1, r_3)\le d(r_1, r_2) + d(r_2, r_3)$.
  \end{description}
  \begin{proof}
    \label{proof:pseudometricd}
    \noindent
    \begin{description}
      \item[\textbf{Nonnegativity}] follows from the fact that weight is constructed from a sum of 0s and 1s.
      \item[\textbf{Symmetry}] follows from the symmetrical construction of the distance formula, which is unchanged by exchanging the two $\relX$ objects.
      \item[\textbf{Reflexivity}] if $g$ and $g'$ are identity functions, then both weights are zero.
    \item[\textbf{Triangle inequality}] calls for a somewhat more intricate argument!
    \end{description}

    Let $g_1\colon Y\to Y'$ and $g_2\colon Y'\to Y''$ be given and let $(r_{x,y})$ be the binary matrix representation of a relation $(X,Y,R)$. First, observe that in the weight of any function $g$, the summation term
    \begin{equation*}
      \sum_{y\in Y}\! \left\{\begin{array}{l}0\ \text{if}\ ( (x,y)\in R\ \text{and}\ (x, g(y))\in R')\\\phantom{0}\ \text{or}\ ((x,y)\notin R\ \text{and}\ (x,g(y)\notin R'))\\1\ \text{otherwise}\end{array} \!\!\right\} = \sum_{y\in Y} |(r'_{x,g(y)}) - (r_{x,y})|.
    \end{equation*}
    Since the RHS expresses a 1-norm between vectors, it inherits the triangle inequality. So we have
    \begin{align*}
      \sum_{y\in Y} |(r''_{x,(g_2\circ g_1)(y)}) - (r_{x,y})| &= \sum_{y\in Y} |(r''_{x,(g_2\circ g_1)(y)}) - (r_{x,y}) + (r'_{x,g_1(y)}) - (r'_{x,g_1(y)})|\\
      &\le \sum_{y\in Y} |(r''_{x,(g_2\circ g_1)(y)}) - (r'_{x,g_1(y)})| + \sum_{y\in Y} |(r'_{x,g_1(y)}) - (r_{x,y})|\\
      &\le \sum_{y\in Y'} |(r''_{x,g_2(y')}) - (r'_{x,g_1(y')})| + \sum_{y\in Y} |(r'_{x,g'(y)}) - (r_{x,y})|. \tag{1}
    \end{align*}

    Next, to show that
    \begin{equation*}
      \#(Y''\setminus (g_2\circ g_1)(Y)) \le \#(Y''\setminus g_2(Y')) + \#(Y\setminus g_1(Y))
    \end{equation*}
    consider an element in the image of $Y'$ under $g_2$ that is not in the image of $Y$ under $g_2\circ g_1$, i.e., $y''\in g_2(Y')\setminus (g_2\circ g_1)(Y)$. Its preimage cannot overlap the image of $Y$ under $g_1$, i.e., $g_2^{-1}(y'')\cap g_1(Y) = \emptyset$. Therefore every element of $g_2^{-1}(y'')$ is in $Y'\setminus g_1(Y)$. So, since $g_1$ is a function, every element of $Y'$ must correspond to exactly one element of $Y''$ and hence
    \begin{equation*}
      \#(Y'\setminus g_1(Y)) \ge \#(g_2(Y')\setminus (g_2\circ g_1)(Y)). \tag{2}
    \end{equation*}
    But since $(g_2\circ g_1)(Y)\subseteq g_2(Y')$, we have $(Y''\setminus g_2(Y')) \subseteq (Y''\setminus (g_2\circ g_1)(Y))$, which in terms of cardinalities means
    \begin{equation*}
      \#(Y''\setminus (g_2\circ g_1)(Y)) - \#(Y''\setminus g_2(Y')) \le \#((g_2\circ g_1)(Y)\subseteq g_2(Y')).
    \end{equation*}
    Combined with (2) and rearranging this means that
    \begin{equation*}
      \#(Y''\setminus (g_2\circ g_1)(Y)) \le \#(Y''\setminus g_2(Y')) + \#(Y\setminus g_1(Y)). \tag{3}
    \end{equation*}
    Taking the maximum over $x\in X$ of the sum of (1) and (3) preserves their inequality and so the weight function exhibits the triangle inequality.

    What remains is to show that the triangle inequality is preserved by the distance function's maximization of the minimization of the weights of $g_3$ and of $g_1$ and $g_2\circ g_1$. For the minimization, observe that every composition $g_2\circ g_1$ is a function $Y\to Y''$, but not conversely. Therefore the minimum over all functions $g_2\colon Y\to Y''$ is taken over a larger set (by inclusion) than the minimums taken separately over functions $g_1\colon Y\to Y'$ and $g_2\colon Y'\to Y''$ and cannot be larger than their sum. The minimization thus preserves the triangle inequality, as, again, does the final maximization. Thus the distance function satisfies the triangle inequality.
  \end{proof}
\end{proposition}

\begin{proposition}
  The pseudometric $d$ respects isomorphism classes of ${\bf Rel}_X$; all elements of an isomorphism class are distance zero apart.
\end{proposition}
\begin{proof}
  If $(Y,R)$ and $(Y',R')$ are isomorphic in ${\bf Rel}_X$, this means that there are functions $\phi: Y \to Y'$ and $\psi = \phi^{-1}$ such that $(x,y) \in R$ if and only if $(x,\phi(y)) \in R'$ for all $x$.  Choosing these as the functions in the definition of $d((Y,R),(Y',R'))$ results in all zeros for both sums, which is the minimum in both cases.  Therefore, $d((Y,R),(Y',R'))=0$.
\end{proof}

\begin{proposition}
  \label{prop:same_cardinality}
  If $Y$ and $Y'$ are the same cardinality, then $d\left((Y,R),(Y',R')\right)$ reduces to
  \begin{equation*}
    d\left((Y,R),(Y',R')\right) = \max \left\{\begin{aligned}
    \min_{\phi : Y \to Y' \atop \text{ bijective}}\max_{x \in X} \sum_{y\in Y} \left\{\begin{aligned}0 & \text{ if }((x,y)\in R \text{ and } (x,\phi(y)) \in R')\\& \text{ or } ((x,y) \notin R \text{ and } (x,\phi(y)) \notin R')\\ 1 &\text{ otherwise}\end{aligned}\right\},\\
    \min_{\psi : Y' \to Y \atop \text{ bijective}}\max_{x \in X}\sum_{y'\in Y'} \left\{\begin{aligned}0 & \text{ if }((x,y')\in R' \text{ and } (x,\psi(y')) \in R)\\& \text{ or } ((x,y') \notin R' \text{ and } (x,\psi(y')) \notin R)\\ 1 &\text{ otherwise}\end{aligned}\right\}
      \end{aligned}\right\}.
  \end{equation*}
\end{proposition}

Specifically, the Proposition gives a condition for when the terms $\# Y' \backslash \phi (Y)$ and  $\# Y \backslash \psi(Y')$ can be omitted when computing the distance.

\begin{proof}
  Suppose that $\phi': Y \to Y'$ is the minimizer among all functions of
  \begin{equation*}
    \label{eq:foo}
    \max_{x \in X}\left\{ \#\{Y' \backslash \phi'(Y)\} + \sum_{y\in Y} \left\{\begin{aligned}0 & \text{ if }((x,y)\in R \text{ and } (x,\phi'(y)) \in R')\\& \text{ or } ((x,y) \notin R \text{ and } (x,\phi'(y)) \notin R')\\ 1 &\text{ otherwise}\end{aligned}\right\}\right\}.
  \end{equation*}
  Since $Y$ and $Y'$ have the same cardinality, there is at least one bijection between them.  If $\phi$ the minimizer among all bijections of
  \begin{equation*}
    \max_{x\in X} \left\{\sum_{y\in Y} \left\{\begin{aligned}0 & \text{ if }((x,y)\in R \text{ and } (x,\phi(y)) \in R')\\& \text{ or } ((x,y) \notin R \text{ and } (x,\phi(y)) \notin R')\\ 1 &\text{ otherwise}\end{aligned}\right\}\right\},
  \end{equation*}
  we claim that $\phi$ and $\phi'$ achieve the same value in the above expressions. Because $\phi: Y \to Y'$ is bijective, this means that
  \begin{equation*}
    \#\{Y' \backslash \phi(Y)\} =0.
  \end{equation*}
  Additionally, the image of $\phi$ contains the image of $\phi'$, since bijections are necessarily surjective.  We may therefore select a bijection $\phi$ that maximizes its agreement with $\phi'$.  Specifically, let us choose a bijection $\phi$ such that if $z$ is in the image of $\phi'$, then there is exactly one $y \in \phi'^{-1}(z)$ such that $\phi(y) = \phi'(y) = z$.  For such a $\phi$, the subset of $Y$ on which $\phi$ and $\phi'$ disagree is at least as large as the number of elements outside the image of $\phi'$, or in other words,
  \begin{equation*}
    \label{eq:bar}
    \#\{y \in Y : \phi(y) \not= \phi'(y) \} \le \# \{Y' \backslash \phi'(Y)\}.
  \end{equation*}

  Consider the $x\in X$ that achieves the maximum value in expression \eqref{eq:foo}.  We can split the sum in that expression into three terms:
  \begin{equation*}
    \begin{aligned}
      \#\{Y' \backslash \phi'(Y)\} &+ \sum_{y : \phi(y)=\phi'(y)} \left\{\begin{aligned}0 & \text{ if }((x,y)\in R \text{ and } (x,\phi'(y)) \in R')\\& \text{ or } ((x,y) \notin R \text{ and } (x,\phi'(y)) \notin R')\\ 1 &\text{ otherwise}\end{aligned}\right\} \\&+ \sum_{y : \phi(y)\not=\phi'(y)} \left\{\begin{aligned}0 & \text{ if }((x,y)\in R \text{ and } (x,\phi'(y)) \in R')\\& \text{ or } ((x,y) \notin R \text{ and } (x,\phi'(y)) \notin R')\\ 1 &\text{ otherwise}\end{aligned}\right\}
    \end{aligned}
  \end{equation*}
  Notice that the middle term is shared by the corresponding sum for $\phi$.

  Moreover because of the inequality \eqref{eq:bar},
  \begin{equation*}
    \sum_{y : \phi(y)\not=\phi'(y)} \left\{\begin{aligned}0 & \text{ if }((x,y)\in R \text{ and } (x,\phi(y)) \in R')\\& \text{ or } ((x,y) \notin R \text{ and } (x,\phi(y)) \notin R')\\ 1 &\text{ otherwise}\end{aligned}\right\} \le \#\{Y' \backslash \phi'(Y)\}.
  \end{equation*}
  Therefore,
  \begin{align*}
    \sum_{y\in Y} & \left\{ \begin{aligned}0 & \text{ if }((x,y)\in R \text{ and } (x,\phi(y)) \in R')\\& \text{ or } ((x,y) \notin R \text{ and } (x,\phi(y)) \notin R')\\ 1 &\text{ otherwise}\end{aligned}\right\}\\ &\le \#\{Y' \backslash \phi'(Y)\} + \sum_{y : \phi(y)=\phi'(y)} \left\{\begin{aligned}0 & \text{ if }((x,y)\in R \text{ and } (x,\phi'(y)) \in R')\\& \text{ or } ((x,y) \notin R \text{ and } (x,\phi'(y)) \notin R')\\ 1 &\text{ otherwise}\end{aligned}\right\} \\
    &\le \#\{Y' \backslash \phi'(Y)\} + \sum_{y : \phi(y)=\phi'(y)} \left\{\begin{aligned}0 & \text{ if }((x,y)\in R \text{ and } (x,\phi'(y)) \in R')\\& \text{ or } ((x,y) \notin R \text{ and } (x,\phi'(y)) \notin R')\\ 1 &\text{ otherwise}\end{aligned}\right\} \\
    & \qquad +\sum_{y : \phi(y)\not=\phi'(y)} \left\{\begin{aligned}0 & \text{ if }((x,y)\in R \text{ and } (x,\phi'(y)) \in R')\\& \text{ or } ((x,y) \notin R \text{ and } (x,\phi'(y)) \notin R')\\ 1 &\text{ otherwise}\end{aligned}\right\}
  \end{align*}
  This implies that if $\phi'$ is the minimizer as claimed, then there is a bijection taking the same value.
\end{proof}

It is important to be wary of the distinction between morphisms and functions in the definition of the pseudometric $d$. Specifically, although the pseudometric can measure the distance between $R_1 = \begin{pmatrix} 1 \end{pmatrix}$ and $R_2 = \begin{pmatrix} 0 \end{pmatrix}$, there are no morphisms at all from $R_1 = \begin{pmatrix} 1 \end{pmatrix}$ to $R_2 = \begin{pmatrix} 0 \end{pmatrix}$.

Observing that distance is well defined for all pairs of $\relX$ objects, one might ask whether there will always be a $\relX$ morphism between two $\relX$ objects, whose weight is no greater than the distance between the objects? It turns out that the answer is No, as the following two examples demonstrate.

% Given $\relX$ objects $r_1=(Y_1,R_1)$ and $r_2=(Y_2,R_2)$ for which $d(r_1,r_2)=k$, does there always exist a morphism $h\colon r_1\to r_2$ with $w(h)\le k$?

\begin{example}
  \label{eg:eg_nomorph_emptyrel}
  Consider the case of $\relX$ objects $r_1=(Y, R_1), r_2=(Y, R_2)$ with single-element sets $X=\{1\},Y=\{a\}$ and relations $R_1=\{(1,a)\}$ and $R_2=\emptyset$. No function $h\colon Y\to Y$ can give rise to a valid $\relX$ morphism $(Y,R_1)\to (Y,R_2)$, because the definition of a $\relX$ morphism requires that $(x,h(y))\in R_2$ whenever $(x,y)\in R_1$ but this is impossible since $R_2=\emptyset$. However, because weight and distance metrics are defined in terms of any function on $Y$, not just those honoring relations (i.e., not just morphisms), and because for any single-element $Y$ the only function $Y\to Y$ is the identity function $id_{Y}$, we can calculate
  \begin{align*}
    w(id_{Y}|r_1,r_2) &= 0 + \left\{\begin{array}{l}0\ \text{if}\ (1,a)\in \{(1,a)\}\ \text{and}\ (1,a)\in \emptyset\dots\\1\ \text{otherwise}\end{array}\right\}\\
    &= 0 + 1 = 1\\
    d(r_1,r_2) &= \max\{\displaystyle{\min_{\phi\in\{id_{Y}\}}} \{w(\phi|r_1,r_2)\}\}\\
    &= w(id_{Y}|r_1,r_2)\\
    &= 1
  \end{align*}
\end{example}

\begin{example}
  \label{eg:eg_nomorph_diffrels}
  For a second example, consider $\relX$ objects $R_1, R_2$ defined by $X=\{a,b,c,d\},\ Y_1=\{1,2,3,4\}$, and $Y_2=\{5,6,7\}$, and relations
  \begin{equation*}
    R_1 = \begin{pmatrix} 1 & 1 & 1 & 0\\1 & 0 & 0 & 0\\0 & 1 & 1 & 1\\0 & 0 & 1 & 1\end{pmatrix} \quad\text{and}\quad R_2 = \begin{pmatrix} 1 & 1 & 0\\1 & 0 & 1\\0 & 1 & 0\\0 & 1 & 1\end{pmatrix}
  \end{equation*}
  labeling rows $X$ and columns $Y_1, Y_2$ from top left. Calculating the distance requires considering $3^4=81$ functions $\phi\colon Y_1\to Y_2$ and $4^3=64$ functions $\psi\colon Y_2\to Y_1$. As may be verified by the reader,\footnote{This may be done by an exhaustive search, or, for instance, by following the strategy to prove lemma \ref{lemma:kappabound}, outlined in section \ref{partitionform} at page \pageref{proofstrategy}.} however, the functions given by
  \begin{equation*}
    \phi(1)=5,\ \phi(2)=6,\ \phi(3)=6,\ \phi(4)=7
  \end{equation*} and \begin{equation*}
    \psi(5)=2,\ \psi(6)=3,\ \psi(7)=1
  \end{equation*}
  yield the required minimum values. This means that
  \begin{equation*}
    \#(Y_2\!\setminus\!\phi(Y_1)) = 0\quad\text{and}\quad \#(Y_1\!\setminus\!\psi(Y_2)) = 1.
  \end{equation*}
  As prefigured in the proof of Proposition \ref{prop:disapseudometric}, we can calculate the second term in the weight of a function by comparing the binary matrix of its domain to a matrix constructed of columns from its range indexed by the function. For $\phi$ this matrix is:
  \begin{equation*}
    (R_2)_{x,\phi(y_1)} = \begin{pmatrix}1 & 1 & 1 & 0\\1 & 0 & 0 & 0\\0 & 1 & 1 & 1\\0 & 0 & 1 & 1\end{pmatrix}
  \end{equation*}
  and for $\psi$ it is:
  \begin{equation*}
    (R_1)_{x,\psi(y_2)} = \begin{pmatrix}1 & 1 & 1\\0 & 0 & 1\\1 & 1 & 0\\0 & 1 & 0\end{pmatrix}.
  \end{equation*}
  These matrices differ from $(R_1)_{x,y_1}$ and $(R_2)_{x,y_2}$, respectively, at most one time per row. Hence the distance is
  \begin{equation*}
    d(r_1, r_2) = \max\{0 + 1, 1 + 1\} = 2.
  \end{equation*}

  Now, $\phi$ does not define a valid morphism $r_1 \to r_2$, because $(c,\phi(4)=7)\notin R_2$. We can, however, construct an alternative $\phi'$ that is a morphism by substituting $\phi'(4)=6$, which yields the same maximum difference per row of 1 and hence the same distance $d(r_1,r_2)$. In the other direction, $\psi$ also fails to define a valid morphism because, \emph{inter alia}, $(d,\psi(7)=1)\notin R_1$. No substitution can render $\psi$ a proper morphism, however---indeed, no morphism is possible in this direction---because there is no $y\in Y_1$ for which both $(b,y)\in R_1$ and $(d,y)\in R_1$.
\end{example}

As a bit of a preview, Figure \ref{fig:eg_nomorph_diffrels} illustrates the Dowker complexes $D(r_1)$ and $D(r_2)$ for the $\relX$ objects $r_1$ and $r_2$ of example \ref{eg:eg_nomorph_diffrels}.  (See definition \ref{df:dowker} for \emph{Dowker complex}.) 

\begin{figure}[ht]
  \includegraphics[height=1.5in]{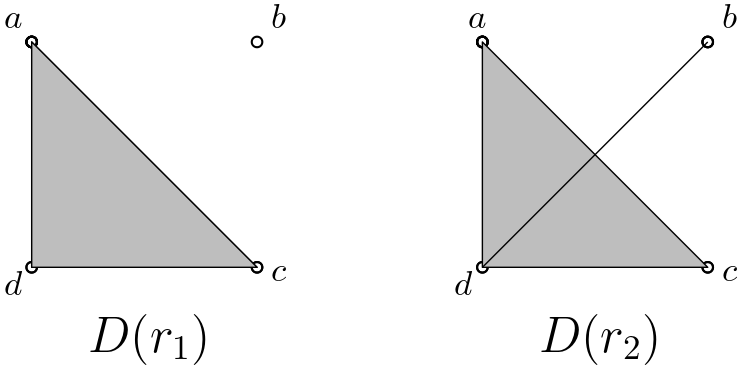}
  \caption{Dowker complexes $D(r_1)$ and $D(r_2)$ for Example \ref{eg:eg_nomorph_diffrels}.}
  \label{fig:eg_nomorph_diffrels}
\end{figure}

The preceding two examples demonstrated that we can calculate distance between $\relX$ relations regardless of whether a $\relX$ morphism exists between them, and both a morphism and a non-morphism can give rise to the same weight and distance. Example \ref{eg:eg_nomorph_emptyrel} showed that there can fail to be a morphism if one of the $\relX$ objects has an empty relation. What if neither relation is empty? Is there always at least one morphism between two objects $(Y_1,R_1), (Y_2,R_2)$ of $\relX$ if neither has an empty relation, i.e., so long as $R_1\neq\emptyset\neq R_2$? As the following example demonstrates, the answer is No.

\begin{example}
  \label{eg:eg_nomorph_bidir}
  Given $X=\{a,b,c\}$ and $Y=\{1,2\}$ consider the pair of $\relX$ objects $r_1=(Y,R_1)$ and $r_2=(Y,R_2)$ with
  \begin{equation*}
    R_1=\begin{pmatrix} 1 & 0\\ 1 & 0\\ 0 & 1\\ \end{pmatrix}\quad\text{and}\quad R_2=\begin{pmatrix} 1 & 0\\ 0 & 1\\ 0 & 1\\ \end{pmatrix}
  \end{equation*}
  There can be no morphism $r_1\to r_2$ because there is no $y'\in Y_2$ for which both $(a,y')$ and $(b,y')$, nor any morphism in the other direction $r_2\to r_1$ because there is no $y\in Y_1$ for which both $(b,y)$ and $(c,y)$.
\end{example}

\section{Dowker Complexes}
\label{sec:dowker}

The discussion so far demonstrates that we need to be careful when applying weight and distance to $\relX$ objects, because we can calculate these quantities for pairs of $\relX$ objects, regardless of whether valid $\relX$ morphisms exist between them.  It is already known that the Dowker complex is a covariant functor from $\rel$ to the category of abstract simplicial complexes \cite[12, p. 9 (Theorem 3)]{robinson_cosheaf_2020}. Although the functor is not faithful---non-isomorphic relations can have the same Dowker complex see \cite[12, p. 11 (explaining how a $\rel$ morphism can change both $X$ and $Y$ without changing the Dowker complex)]{robinson_cosheaf_2020}---can we guarantee the existence of a $\relX$ morphism between two $\relX$ objects if we restrict ourselves to objects that have a valid simplicial map (Definition \ref{df:simplicialmap}) between their Dowker complexes? It turns out that the answer is, No, and, in fact, the implication proceeds in the converse direction: The presence of a $\relX$ morphism ensures the existence of a valid simplicial map between Dowker complexes.

\begin{definition}
  An \emph{abstract simplicial complex $X$ on a set $V_X$} consists of a set $X$ of subsets of $V_X$ such that if $\sigma \in X$ and $\tau \subseteq \sigma$, then $\tau \in X$.  Each $\sigma \in X$ is called a \emph{simplex of $X$}, and each element of $V_X$ is a \emph{vertex of $X$}.  Every subset $\tau$ of a simplex $\sigma$ is called \emph{face} of $\sigma$.
\end{definition}

It is usually tiresome to specify all of the simplices in a simplicial complex.  Instead, it is much more convenient to supply a \emph{generating set} $S$ of subsets of the vertex set.  The unique smallest simplicial complex containing the generating set is called the \emph{abstract simplicial complex generated by $S$}.

\begin{definition}
  \label{df:dowker}
  The \emph{Dowker complex} $D(X,Y,R)$ is the abstract simplicial complex on vertex set $X$ whose simplices $\alpha=[x_{0}, \dots, x_{k}]$ are composed of vertices all of which share a relation in $R$ with the same $y\in Y$:
  \begin{equation*}
    D(X,Y,R) = \{[x_{0}, \dots, x_{k}]\colon\ \text{there exists a}\ y\in Y\ \text{s.t.}\ (x_{i}, y)\in R\ \text{for all}\ i=0, \dots, k\}.
  \end{equation*}
  The \emph{total weight} \cite{robinson_cosheaf_2020} is a function $t: D(X,Y,R)\to \mathbb{N}$ given by
  \begin{equation*}
    t_R(\sigma) = \# \{y \in Y : (x,y) \in R\text{ for all }x \in \sigma\}.
  \end{equation*}
  The \emph{differential weight} \cite{Ambrose_2020} is a function $d : D(X,Y,R) \to \mathbb{N}$ given by
  \begin{equation*}
    d_R(\sigma) = \# \{y \in Y: \left((x,y)\in R\text{ if } x\in \sigma\right) \text{ and } \left((x,y)\notin R\text{ if } x\notin \sigma\right) \}.
  \end{equation*}
  When only one relation is being discussed, we will often write $t(\sigma)$ for $t_R(\sigma)$ and $d(\sigma)$ for $d_R(\sigma)$.
\end{definition}

For convenience, let us begin by defining
\begin{equation*}
  Y_\sigma = \{y\in Y : (x,y) \in R\text{ for all }x\in\sigma\}
\end{equation*}
for a simplex $\sigma$ of $D(X,Y,R)$.  The total weight function is simply the cardinality of this set: $t(\sigma) = \# Y_\sigma$.

\begin{remark}
    Is the total weight a discrete Morse function?  It is not; consider the following example.  Let $X=\{A,B,C\}$, $Y=\{a\}$, and $R$ be defined by the matrix
    \begin{equation*}
      R = \begin{pmatrix} 1 \\ 1 \\ 1 \end{pmatrix}.
    \end{equation*}
    In this case, the Dowker complex $D(X,Y,R)$ is the complete simplex $[A,B,C]$.  The total weight on every simplex is the same, namely $1$.  This violates the discrete Morse function condition.
\end{remark}

The pseudometric we defined on ${\bf Rel}_X$ is compatible with the total weight functions for the Dowker complex developed previously \cite{robinson_cosheaf_2020}.

\begin{proposition}
  \label{prop:weight_bound}
  The pseudometric $d$ bounds the difference in total weight on any face of the Dowker complex---unused faces given a weight of zero.  Specifically, consider two relations $R$ and $R'$ on the same pair of sets $X$ and $Y$.  Then for each simplex $\sigma \in 2^X$, the total weight $t_R(\sigma)$ can differ from the total weight of $t_{R'}(\sigma)$ by the following,
  \begin{equation*}
    |t_R(\sigma)-t_{R'}(\sigma)| \le d((Y,R),(Y,R')) (\dim (\sigma)+1)
  \end{equation*}
  where the total weight is taken to be zero if $\sigma$ is not in one of the Dowker complexes.  Moreover, the same bound holds for the differential weight functions as well, namely
  \begin{equation*}
    |d_R(\sigma)-d_{R'}(\sigma)| \le d((Y,R),(Y,R')) (\dim (\sigma)+1).
  \end{equation*}
\end{proposition}
The fact that the same bound works for both weights indicates that the bound is often loose for one or the other kind of weight function.
\begin{proof}
  If $d((Y,R),(Y,R'))=k$, this means that there may be at most $k$ differences in each row of the matrices for $R$ and $R'$.  These differences need not occur in the same columns in each row.  Since $\sigma \in 2^X$ corresponds to $(\dim (\sigma)+1)$ rows (elements of $X$), then a maximum of $k (\dim(\sigma)+1)$ columns may differ when we restrict attention to the columns related to $\sigma$.  Therefore, the total weight for $\sigma$ may not differ more than that amount.  Since the differential weight counts exact matches rather than inclusions of columns, at most $k (\dim(\sigma)+1)$ columns may differ, so this is the most the differential weight may change on any given simplex.
\end{proof}

Combining Proposition \ref{prop:weight_bound} with Proposition \ref{prop:disapseudometric} and the fact that the weight functions are complete isomorphism invariants yields the following pleasing result.

\begin{corollary}
  The pseudometric $d$ is a metric on isomorphism classes of $\cat{Rel}$.  That is, two relations are isomorphic if and only if their $d$-distance is zero.
\end{corollary}

\begin{example}
  \label{eg:eg_tight_dist_bound}
  The bound given in Proposition \ref{prop:weight_bound} is tight, which is to say that there are pairs of relations whose weight functions attain it.  As an example, consider the relations $R_1$ and $R_2$ on $X+\{a,b\}$, $Y=\{1,2\}$ given by
  \begin{equation*}
    R_1 = \begin{pmatrix}
      1 & 1\\
      1 & 1\\
      \end{pmatrix}
  \end{equation*}
  and
  \begin{equation*}
    R_2 = \begin{pmatrix}
      1 & 0\\
      0 & 1\\
      \end{pmatrix}.
  \end{equation*}
  It is easy to see that $d((Y,R_1),(Y,R_2)) = 1$, since regardless of the function $\phi$ or $\psi$ chosen to transform columns, at least $1$ difference occurs in each row.  The total and differential weights of $[a,b]$ are $d_{R_1}([a,b]) = 2$, $t_{R_1}([a,b]) = 2$, $t_{R_2}([a,b]) = 0$, and $d_{R_2}([a,b]) = 0$.  These observations agree with Proposition \ref{prop:weight_bound} because $d((Y,R_1),(Y,R_2)) (\dim ([a,b])+1) = 2$.
\end{example}

\begin{figure}[ht]
  \includegraphics[height=1.5in]{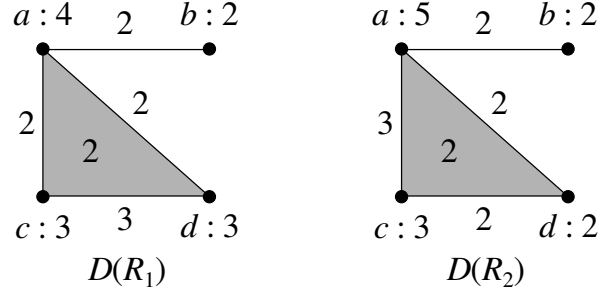}
  \caption{The Dowker complexes $D(R_1)$ and $D(R_2)$ and total weight functions $t_{R_1}$ and $t_{R_2}$ for Example \ref{eg:eg_total_weight}.}
  \label{fig:eg_total_weight}
\end{figure}

\begin{example}
  \label{eg:eg_total_weight}
  Consider the two relations $R_1$ and $R_2$ given by the following matrices:
  \begin{equation*}
    R_1 = \begin{pmatrix}
      1&1&1&1&0\\
      1&1&0&0&0\\
      0&0&1&1&1\\
      0&0&1&1&1\\
    \end{pmatrix}, \;
    R_2 =\begin{pmatrix}
    1&1&1&1&1\\
    1&1&0&0&0\\
    0&0&1&1&1\\
    0&0&0&1&1\\
    \end{pmatrix}.
  \end{equation*}
  A somewhat arduous calculation (made a bit easier by Proposition \ref{prop:same_cardinality}) reveals that $d(R_1,R_2)=1$.  If we index the rows by $X=\{a,b,c,d\}$, then the two Dowker complexes and corresponding total weights on each simplex are shown in Figure \ref{fig:eg_total_weight}.  The Figure shows that the total weights differ on four simplices, and each by not more than $1$.  This is in accordance with Proposition \ref{prop:weight_bound}, since the smallest dimension (that of a vertex) is zero, the total weights may differ by $1$.
\end{example}

The morphisms in the category of abstract simplicial complexes are \emph{simplicial maps}.  These maps are the specialization of continuous maps to the context of abstract simplicial complexes, and as such admit a combinatorial definition.

\begin{definition}
  \label{df:simplicialmap}
  Given abstract simplicial complexes $X$ and $Y$ with vertex sets $V_{X}$ and $V_{Y}$, respectively, a function $f\colon V_{X}\to V_{Y}$ is a \emph{simplicial map} if it maps every simplex $[v_0,\dots,v_k]\in X$ to a simplex $[f(v_0),\dots,f(v_k)]\in Y$ after removing duplicate vertices.
\end{definition}

The relationship between $\relX$ morphisms and simplicial maps is somewhat tenuous.  As further consideration of Example \ref{eg:eg_nomorph_bidir} shows, the existence of a simplicial map does not guarantee the existence of a morphism in $\relX$.  Specifically, given $\relX$ objects $(Y_1,R_1)$ and $(Y_2,R_2)$ with a simplicial map $\phi$ such that $\phi(D(X,Y_1,R_1))=D(X,Y_2,R_2)$, is there always at least one morphism $h\colon(Y_1,R_1)\to(Y_2,R_2)$?  Example \ref{eg:eg_nomorph_smplclmap} answers this question in the negative.

\begin{example}
  \label{eg:eg_nomorph_smplclmap}
  Recall that this Example involved the pair of $\relX$ objects $r_1=(Y,R_1)$ and $r_2=(Y,R_2)$ with $X=\{a,b,c\},\ Y=\{1,2\}$ and
  \begin{equation*}
    R_1=\begin{pmatrix} 1 & 0\\ 1 & 0\\ 0 & 1\\ \end{pmatrix}\quad\text{and}\quad R_2=\begin{pmatrix} 1 & 0\\ 0 & 1\\ 0 & 1\\ \end{pmatrix}.
  \end{equation*}
  Passing them through the Dowker functor yields
  \begin{equation*}
    d_1 = D(r_1) = \{[a],[b],[c],[a,b]\}\quad \text{and}\quad d_2 = D(r_2) = \{[a],[b],[c],[b,c]\}.
  \end{equation*}
  As illustrated in figure \ref{fig:eg_nomorph_smplclmap} the vertex map
  \begin{equation*}
    \phi(a)=c,\ \phi(b)=b,\ \phi(c)=a
  \end{equation*}
  results in a simplicial map $\phi(D(r_1))=D(r_2)$. But as demonstrated before, there is no morphism $h\colon r_1\to r_2$.
\end{example}

\begin{figure}[ht]
  \includegraphics[height=1.5in]{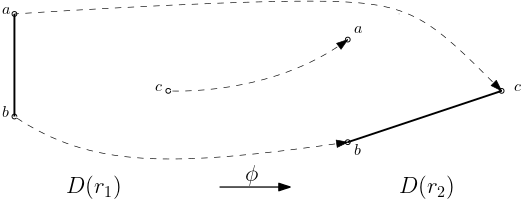}
  \caption{Dowker complexes $D(r_1)$ and $D(r_1)$ with simplicial map for Examples \ref{eg:eg_nomorph_smplclmap} and \ref{eg:eg_differentmaximalsimplices}.}
  \label{fig:eg_nomorph_smplclmap}
\end{figure}

In contrast, the existence of any $\relX$ morphism does imply the existence of a simplicial map.

\begin{proposition}
  \label{proposition:Dowkersimplicialmapexistence}
  If $(Y_1,R_1)$ and $(Y_2,R_2)$ are objects in $\relX$ with a morphism $h\colon(Y_1,R_1)\to(Y_2,R_2)$ then there is at least one simplicial map $\phi$ such that $\phi(D(X,Y_1,R_1))=D(X,Y_2,R_2)$.
\end{proposition}

\begin{proof}
  Given $r_1=(Y_1,R_1)$ and $r_2=(Y_2,R_2)$ in $\relX$ with a valid morphism $f\colon r_1\to r_2$, let $d_1=D(X,Y_1,R_1)$ and $d_2=D(X,Y_2,R_2)$ and $\phi$ be a map $(X,Y_1,R_1)\to (X,Y_2,R_2)$ with $id_X$ as the vertex map. Then $\phi$ is a simplicial map, because for every simplex $\sigma\in d_1$ we have $(x,y)\in Y_1$ for all vertices $x\in\sigma$ and $(\phi(x)=x, f(y))\in R_2$, i.e., whenever $\sigma=[x_1,\dots,x_k]\in d_1$ we have $\tau=[\phi(x_1),\dots,\phi(x_k)]\in d_2$.
\end{proof}

This result, of course, is consistent with the fact that the Dowker complex is a covariant functor from $\rel$ to the category of abstract simplicial complexes $\cat{Asc}$, which means that morphisms in $\rel$, and hence also subcategory $\relX$, correspond to morphisms in $\cat{Asc}$, which are simplicial maps (see \cite[]{robinson_cosheaf_2020}).

What does the existence of a simplicial map $\phi\colon D(X,Y_1,R_1)\to D(X,Y_2,R_2)$ without a $\relX$ morphism $(X,Y_1,R_1)\to (X,Y_2,R_2)$ tell us? Dowker complexes with a simplicial map having $id_x$ as the vertex map but no $\relX$ morphism differ in their maximal simplices.

\begin{proposition}
  \label{proposition:Dowkeridentitysimplicialmap}
  If $(Y_1,R_1)$ and $(Y_2,R_2)$ are objects in $\relX$ without any morphism $h\colon(Y_1,R_1)\to(Y_2,R_2)$ but $\phi = (id_X, f)$ is a simplicial map $\phi\colon D(X,Y_1,R_1) \to D(X,Y_2,R_2)$, then the sets of maximal simplices of the two Dowker complexes are different.
\end{proposition}

\begin{proof}
  The simplicial map with $id_X$ as the vertex map means that $(x,y')\in R_2$ for some $y'\in Y_2$ whenever $(x,y)\in R_1$ for some $y\in Y_2$, but the absence of a $\relX$ morphism means no function $f(y)=y'$ can be defined. Consequently, there is some set of vertices $x\in X$ in relation to the same $y\in Y_1$, not all of whose elements are in relation to any one $y'\in Y_2$, i.e., there exists some $\sigma=[x_1,\dots,x_k]\in D(X,Y_1,R_1)$ for which
  \begin{equation*}
    \bigcap_{x\in\sigma}\{y'\in Y_2\colon (x,y')\in R_2\}=\emptyset \tag{1}
  \end{equation*}
  (Note that the simplex $\sigma$ must be of dimension 1 or greater.) Now, every simplex in a finite abstract simplicial complex must be contained within some maximal simplex in the complex. Hence $\sigma$ is contained in some maximal simplex $\tau \supseteq \sigma,\ \tau \in D(X,Y_1,R_1)$. Yet (1) means that no maximal simplex in $D(X,Y_2,R_2)$ can involve all of the vertices in $\sigma$. Hence the respective sets of maximal simplices of the two Dowker complexes must be different.
\end{proof}

\begin{example}
  \label{eg:eg_differentmaximalsimplices}
  Example \ref{eg:eg_nomorph_smplclmap} showed two relations $r_1, r_2$ between which no $\relX$ morphism existed but for which a simplicial map could be defined between their Dowker complexes (see figure \ref{fig:eg_nomorph_smplclmap}). The Dowker complexes of the two relations are
  \begin{equation*}
    D(r_1) = \{[a], [b], [c], [ab]\}\ \text{and}\ D(r_2) = \{[a], [b], [c], [bc]\},
  \end{equation*}
  and, indeed, the maximal simplices are different:
  \begin{equation*}
    [ab] \neq [bc].
  \end{equation*}
\end{example}

Proposition \ref{proposition:Dowkeridentitysimplicialmap}, might not be seem surprising, since Dowker complexes with the same maximal simplices on the same vertex set are obviously identical. The nuance of the proposition is that the absence of a morphism between the underlying $\relX$ objects actually prevents identity of Dowker complexes, despite the lack of faithfulness of the Dowker functor. The presence or absence of a morphism thus seems to point to something significant about the relations, which the presence or absence of simplicial maps between corresponding Dowker complexes does not capture. If we want to understand weight and distance between $\relX$ objects, then, we may want to try a different perspective.

\section{Relation Multiset is a Covariant Functor}
\label{sec:multiset_functor}

An alternative to Dowker complexes is presented by multisets. Note that for any element $y\in Y$ in a $\relX$ object we can collect all $x\in X$ with which $y$ is in relation, i.e., for a given $y$ we can collect

\begin{equation*}
\sigma = \{x\in X\colon (x,y)\in R\ \text{for all}\ x\in\sigma\ \text{and}\ (x,y)\notin R\ \text{for all}\ x\notin\sigma\}.
\end{equation*}

Such a subset of $X$, of course, is the same as the vertices in a
simplex of the Dowker complex $D(X,Y,R)$ for the same $(Y,R)$, but we can take a different perspective: Given $X$ and $\relX$ object $(Y,R)$, it is evident that every $y\in Y$ can be assigned exactly one such $\sigma\subseteq X$ and that collections of $y$ assigned the same $\sigma$ constitute disjoint subsets
\begin{equation*}
  Y^{\sigma}=\{y\in Y\colon (x,y)\in R\ \text{for all}\ x\in\sigma\ \text{and}\ (x,y)\notin R\ \text{for all}\ x\notin\sigma\}
\end{equation*}
whose union equals $Y$, that is, that $Y$ is partitioned by some $\{\sigma_i\}$ into $\{Y^{\sigma_i}\}$.\footnote{The cardinality of
\begin{equation*}
  \#Y^{\sigma}=\#\{y\in Y\colon (x,y)\in R\ \text{for all}\ x\in\sigma\ \text{and}\ (x,y)\notin R\ \text{for all}\ x\notin\sigma\}
\end{equation*}
is equal to the differential weight $d_R(\sigma)$ according to Definition \ref{df:dowker}.} This means $Y$ is a species of multiset.

\begin{definition}
  \label{df:mul}
  Loosely following \cite[]{monro_concept_1987} as extended by \cite[]{singh_mathematics_2016}~(see also, e.g., \cite[]{syropoulos_categorical_2003}), we define the category $\mul$ of \emph{multisets} $A$ whose objects are pairs $<A_{0}, \pi>$ of an ordinary set $A_{0}$ called the \emph{field} and a \emph{classifier} $\pi_{\Omega}\colon A_{0}\to \Omega$ that is injective from the field to an indexing set and identifies equivalence classes / partitions $A^{\sigma}$ by $\sigma\in\Omega$, including $\sigma=\emptyset$. Morphisms $<A_{0}, \pi_{\Omega}>\ \to\ <B_{0}, \pi_{\Psi}>$ are a pair of functions $f\colon A_{0}\to B_{0}$ and $g\colon \Omega\to\Psi$ that \emph{honor partitions}, i.e., if $\pi_{\Omega}(a) = \pi_{\Omega}(a')$ then $\pi_{g(\Omega)}(f(a))=\pi_{g(\Omega)}(f(a'))$. When it is clear what index set we are considering, we will drop it from the classifier's notation.
\end{definition}

\begin{proposition}
\label{proposition:muliswelldefined}
  The category $\mul$ is well defined.
\end{proposition}

\begin{proof}
  Composition of morphisms is simply composition of the morphism functions $f,g$, which a quick check confirms is associative: Given morphisms $(f,g)\colon A\to B,\ (f',g')\colon B\to C$, and $(f'',g'')\colon C\to D$ and $x\in X_0, \sigma\in\Omega$, we have
  \begin{equation*}
    (f'' \circ (f' \circ f))) (x) = f''(f'(f(x))) = (f'' \circ f') (f(x)) = ((f'' \circ f') \circ f) (x)
  \end{equation*}
  and
  \begin{equation*}
    (g'' \circ (g' \circ g))) (\sigma) = g''(g'(g(\sigma))) = (g'' \circ g') (g(\sigma)) = ((g'' \circ g') \circ g) (\sigma). \qedhere
  \end{equation*}
\end{proof}

Given two multisets $A =\ <A_0,\pi_{A}>,B=\ <B_0,\pi_{B}>$ we have the following further definitions and properties: The cardinality of each partition is called its \emph{multiplicity} $\#A^{\sigma}$ and the largest cardinality of all equivalence classes in a multiset is the multiset's \emph{height} $\hgt(A) = \max_{\sigma\in\pi(A)} \#A^{\sigma}$. We use the notation $A^{\sigma}\!:\!n$ to indicate that partition $A^{\sigma}$ has multiplicity $n$. If the classifier $\pi$ is bijective, then all multiplicities are 1 and the multiset is exactly equivalent to an ordinary set. A \emph{submultiset} $B\subseteq A$ is a multiset whose field $B_0\subseteq A_0$ and whose equivalence classes $B^{\sigma_i}$ are each a subset of an equivalence class $A^{\sigma_i}$ of $A$.

\begin{remark}
  \label{remark:Monrocompared}
  This definition differs superficially from Monro's seminal formulation. Rather than a classifier function, Monro uses a relation to specify equivalence classes \cite[13, pp. 171-72]{monro_concept_1987}. Since every classifier induces an equivalence relation $R=\{(x,y)\in X\times Y\colon \pi(x)=\pi(y)\}$ and vice versa, however, our formulation is equivalent. We use a classifier function to emphasize the partitions \emph{per se}, which Monro and others refer to as ``sorts'' (\cite[13, p. 171]{monro_concept_1987}, \cite[21, p. 1141]{syropoulos_categorical_2003}). In essence we label / identify partitions with subsets of $X$, while Monro et alia use an arbitrary representative from each sort. Our definition slightly extends that of Monro and others, who define $\mul$ morphisms solely in terms of a function on the \emph{field}, while we also refer to a function on the \emph{indexing set}. We do this to achieve a generality comparable to that of our $\rel$ category; in practice, however, the extension is mooted by our focus on the $\relX$ category, which imposes identity on $X$ and hence the indexing set.
\end{remark}

\begin{example}
  \label{eg:eg_multiset}
  Recall example \ref{eg:eg_nomorph_diffrels} with $\relX$ objects $r_1, r_2$ defined by $X=\{a,b,c,d\},\ Y_1=\{1,2,3,4\}$, and $Y_2=\{5,6,7\}$, and relations
  \begin{equation*}
    R_1=\begin{pmatrix} 1 & 1 & 1 & 0\\1 & 0 & 0 & 0\\0 & 1 & 1 & 1\\0 & 0 & 1 & 1\end{pmatrix}\quad\text{and}\quad R_2=\begin{pmatrix} 1 & 1 & 0\\1 & 0 & 1\\0 & 1 & 0\\0 & 1 & 1\end{pmatrix}.
  \end{equation*}
  We can restate them in terms of the following multisets:
  \begin{equation*}
    Y_1 =\{Y^{[a,b]}\!:\!1, Y^{[a,c]}\!:\!1, Y^{[a,c,d]}\!:\!1, Y^{[c,d]}\!:\!1\}
  \end{equation*}
  and
  \begin{equation*}
    Y_2 =\{Y^{[a,b]}\!:\!1, Y^{[a,c,d]}\!:\!1, Y^{[b,d]}\!:\!1\},
  \end{equation*}
  recalling that ``$Y^{\sigma}\!:\!m$'' signifies that partition $Y^{\sigma}$ has multiplicity $m$. Since the multiplicity of each of these partitions is 1, the height of both $Y_1$ and $Y_2$ is also just 1. Combining them into a single $r_3$ with
  \begin{equation*}
    R_3 = \begin{pmatrix} 1 & 1 & 1 & 0 & 1 & 1 & 0\\1 & 0 & 0 & 0 & 1 & 0 & 1\\0 & 1 & 1 & 1 & 0 & 1 & 0\\0 & 0 & 1 & 1 & 0 & 1 & 1\end{pmatrix}
  \end{equation*}
  yields
  \begin{equation*}
    Y_3 = \{Y^{[a,b]}\!:\!2, Y^{[a,c]}\!:\!1, Y^{\{a,c,d\}}\!:\!2, Y^{[b,d]}\!:\!1, Y^{\{c,d\}}\!:\!1 \}
  \end{equation*}
  which has height 2.
\end{example}

Since every classifier induces an equivalence relation $R=\{(x,y)\in X\times Y\colon \pi(x)=\pi(y)\}$, the role of relations in both $\rel$ and $\mul$ suggests a close connection between the two categories. Indeed, there is a full and faithful covariant functor between them. When restricted to $\rel_{+}$, this functor is also bijective on objects.

\begin{definition}
  \label{df:mulXmultiset}
  The \emph{relation multiset} $M(X,Y,R)=<Y,\pi>$ is a multiset in which the field is $Y$, the indexing set is the power set of $X$, and the classifier $\pi_{X}\colon Y\to \mathcal{P}(X)$ is defined by
  \begin{equation*}
    \pi_{X}(y) := \{x\in X\colon (x,y)\in R\}.
  \end{equation*}
\end{definition}

As noted at the beginning of this section, each index $\sigma\in\pi_{X}(Y)$ specifies a subset of the field that is disjoint from subsets specified by other indices, i.e., a partition
\begin{equation*}
  Y^{\sigma}=\{y\in Y\colon (x,y)\in R\ \text{for all}\ x\in\sigma\ \text{and}\ (x,y)\notin R\ \text{for all}\ x\notin\sigma\}.
\end{equation*}

\begin{theorem}
  \label{theorem:mulRfunctor}
  Relation multiset $M(X,Y,R)$ is a full and faithful covariant functor $M \colon \rel \to \mul$.
\end{theorem}

\begin{proof}
  \label{proofmulRfunctor}
  Let $r, r'\in Obj(\rel)$ and $m=M(r), m'=M(r')\in Obj(\mul)$. Each $r=(X,Y,R)$ translates directly into a multiset $M(X,Y,R)=<Y_0,\pi_{X}>$, where $Y_0=Y$ and $\pi_{X}(y) := \{x\in X\colon(x,y)\in R\}$. Each $\rel$ morphism $\phi\colon r\to r'$, defined by a $\phi(r) =(f(x), g(y))$, translates directly into a $\mul$ morphism $\mu\colon m\to m'$ defined by $\mu(m) = (g(y),\pi_{f(X)})$ where $\pi_{f(X)}(y') = \{x'\in f(X)\colon (f(x),y')\in R'\}$. The function $\pi_{f(X)}$ honors partitions, as required for a $\mul$ morphism, because the $\rel$ requirement that $(f(x), g(y))\in R'$ whenever $(x,y)\in R$ means that for $a,b\in Y$ whenever $\pi_{X}(a)=\pi_{X}(b)$ we have both $(x,a), (x,b)\in R$. Consequently, $(f(x),g(a)), (f(x),g(b))\in R'$ and hence $\pi_{f(x)}(g(a))=\pi_{f(x)}(g(b))$. Meanwhile, the definition of $\mu$ implies that it inherits identity and covariant composition from $\phi$. Thus the relation multiset is a covariant functor $\rel\to\mul$.

  Faithfulness on morphisms arises because whenever two $\rel$ morphisms are equal, $(f(x), g(y)) = (f'(x), g(y'))$ we have $(g(y), \pi_{f(x)})=(g'(y),\pi_{f'(y)})$.

  Fullness on morphisms arises because, for any $\rel$ morphism $\phi=(f,g)$ between $\rel$ objects $r=(X,Y,R)$ and $r'=(X',Y',R')$ there exists a $\mul$ morphism $\mu:=(g,\pi_{f(X)})$ such that
  \begin{align*}
    M(\phi(r)) &= M((f(X), g(Y), \{(f(x),g(y))\colon (x,y)\in R'\}))\\ &= <g(Y), \pi_{f(X)}>,\ \pi_{f(X)}(g(y)):=\{(f(x),g(y))\colon (x,y)\in R'\}\\ &= \mu(<Y,\pi_{X}>) \qedhere
  \end{align*}
\end{proof}

Relations and relation multisets are not categorially equivalent, because the $M(X,Y,R)$ functor forgets elements of $X$ that are not in relation to any element of the other field. Thus, given a $M(X,Y,R)=< Y_0, \pi_{X} >$, with $x_1\in X$ not in relation to any $Y$, we can reconstruct $Y=Y_0$ and $R=\bigcup_{\sigma\in\pi_{X}(Y_0)} \{(x,y)\colon y\in Y_0\ \text{and}\ x\in\sigma\}$ but $\pi_{X}$ cannot be used to reconstruct the full set $X$ since $x_1\not\in \bigcup_{\sigma\in\pi_{X}}(Y_0) \{x\in\sigma\}$. The functor does not ignore such elements; the classifier merely assigns them all to the same category $\emptyset$ and thereby ``forgets'' the identity of the $x_i$ not in relation to any $y\in Y$. In effect, the multiset functor collapses all zeros rows into one.\footnote{\label{note:forgetful} Thus the category $\mul$ could be adjusted to require a bijective classifier, but full equivalence would take the category yet farther from the seminal definition of Monro and is not needed for the bulk of the analysis in this paper.} When there are no zero rows, there is nothing to forget, and we have the following proposition.

\begin{proposition}[$\rel_{+}$ equivalence]
  \label{proposition:relpequivalence}
  When restricted to $\rel_{+}$, relation multisets are equivalent to relations.
\end{proposition}

\begin{proof}
  \label{proof:relpequivalence}
  Since $M(X,Y,R)$ is a full and faithful functor, we merely need to demonstrate bijectivity on objects. This arises from the identity of the multiset's field $Y_0$ with the relation's $Y$, the identity of the index set with the relation's $X$, and the well-defined nature of the classifier in terms of the relation's $R$. For we can reconstruct $X$ and $R$ with the classifier $\pi$ as follows: Given multiset $<Y_0,\pi_{\Omega}>$, let $Y=Y_0,\ X=\cup\{x\in\sigma\in\pi_{\Omega}(Y_0)\}$, and $R=\cup\{(x,y)\colon y\in Y_0\ \text{and}\ x\in\sigma\in\pi_{\Omega}\}$. Consequently, if $m=<Y_0,\pi_{\omega}>=<Y'_0,\pi_{\Omega'}>=m'$, then \begin{align*} Y = Y_0 &=Y'_0=Y'\\ X =\negthickspace \bigcup_{\pi_{\Omega}(Y_0)}\negthickspace\{x\in\sigma\} &= \negthickspace\bigcup_{\pi_{\Omega'}(Y'_0)}\negthickspace\{x'\in\sigma'\} = X' \tag{*}\\ R =\negthickspace \bigcup_{\pi_{\Omega}(Y_0)}\negthickspace\{(x,y)\colon y\in Y_0\ \text{and}\ x\in\sigma\} &= \negthickspace\bigcup_{\pi_{\Omega'}(Y'_0)} \negthickspace\{(x',y')\colon y'\in Y'_0\ \text{and}\ x'\in\sigma'\} = R'\end{align*} and hence
  \begin{equation*}
    r = (X,Y,R) = (X',Y',R') = (X,Y,R)' = r,
  \end{equation*}
  implying injectivity, while reading the previous display in reverse yields surjectivity. In this regard, we note that the calculation in line (*) recovers all of $X$ because in $\rel_{+}$ there are no $x_i\in X$ that fail to have a relation with any $y\in Y$.
\end{proof}

By definition, objects $r_X=(Y,R)$ in $\relX$ are objects $r=(X,Y,R)$ in $\rel$ restricted to a single set $X$ and all morphisms in $\relX$ have the identity function for the $X$ component. That means that all $\mulX$ morphisms are similarly constrained to $id_X$ and we have a corresponding subcategory $\mulX$. The multiset functor, therefore, can be viewed as enriching the $\rel$ and $\relX$ categories with multiset classifiers derived from their relations $R$. In what follows, we will elide the application of the multiset functor and work directly with $\rel$ and $\relX$ objects enriched in this way.

\section{Calculating Weight and Distance with Partitions}
\label{sec:calculating}

Treating $\relX$ objects as multisets allows us to analyze weight and distance between two relations as functions of their respective partitions. The main result of this section is that we can do so without regard to whether simplicial maps exist between their Dowker complex representations. The formulation in terms of partitions enables us to provide upper bounds for both weight and distance that depend only on easily-computed measures of the two objects' respective relations, without requiring searches of the functional space between them. To derive the calculation and demonstrate the validity of the bounds, however, we do need to consider two cases: according to whether or not a $\relX$ morphism exists between them. The specific bounds for these cases relate to each other but differ.

The plan for this section is as follows: We first define some useful notation; cast weight in terms of partitions and observe some facts useful to the proof; derive a general bound on the weight function; use these facts to calculate exact and bounded estimates of weight and distance; and finally collect interim results in a comprehensive statement.

\subsection{Notation}

For the discussion in this section, we will assume $\relX$ objects $r_1 = (Y_1, R_1)$ and $r_2 = (Y_2, R_2)$ after passing them through the relation multiset functor to enrich them with classifiers from the respective field to the powerset of $X$, $\pi_{Y_1}\colon Y_1\to \mathcal{P}(X)$ and $\pi_{Y_2}\colon Y_2\to \mathcal{P}(X)$ as per definition \ref{df:mulXmultiset}. It is convenient to define a joint classifier $\pi \colon Y_1\cup Y_2 \to\mathcal{P}(X)$ as
\begin{equation*}
  \pi(y) := \begin{cases} \pi_{Y_1}(y) &\text{if}\ y\in Y_1\\\pi_{Y_2}(y) &\text{if}\ y\in Y_2 \end{cases}.
\end{equation*}

Note that this joint classifier is well defined, because if $y$ is in both $Y_1$ and $Y_2$, then $\pi_{Y_1}(y)=\pi_{Y_2}(y)$. Indeed, this collection of equivalence classes found in both $r_1, r_2$ is so useful that we also define
\begin{equation*}
  \pis(Y_1\cup Y_2) = \pi_{Y_1}(Y_1)\cap \pi_{Y_2}(Y_2)
\end{equation*}
to refer to indices of equivalence classes found in both sets; we may use $\pis$ when the $Y_1,Y_2$ are clear. Although it is not necessary for some of the following results, to enable commutativity for certain operations in later sections, we also assume partial orders on the fields $X,Y_1,Y_2$ and perform operations in ways that honor the partial orders.

Finally, central to the discussions will be various subsets of agreeing and disagreeing partitions, for which we establish some additional notation. Continuing the convention that superscripts signify restriction, we will refer to the group of agreeing partitions as $Y_2^{\pis}=\{Y_2^{\sigma}\colon \sigma\in\pis\}$, to the group of disagreeing partitions as $Y_2^{\not\pis}=\{Y_2^{\sigma}\colon \sigma\notin\pis\}$, to the group of disagreeing partitions whose indices include $\hat{x}$ as $Y_2^{\not\pis\owns \hat{x}}$, and to the group of disagreeing partitions whose indices do not include $\hat{x}$ as $Y_2^{\not\pis\not\owns \hat{x}}$. Thus we have
\begin{equation*}
  Y_2 = Y_2^{\pis} + Y_2^{\not\pis} = Y_2^{\pis} + (Y_2^{\not\pis\owns \hat{x}} + Y_2^{\not\pis\not\owns \hat{x}}).
\end{equation*}

\subsection{Partition Form of Weight and Distance}
\label{partitionform}

We now state \emph{weight} in terms of partitions. In general, given $\relX$ objects $r_1, r_2$, the weight of an arbitrary function $g\colon r_1\to r_2$ can be written in terms of partitions as follows:
\begin{align}
  w(g|r_1, r_2) &= \max_{x\in X} \left\{ \#\{Y_2\!\setminus\! g(Y_1)\}\! +\!\! \sum_{y\in Y_1}\! \left\{ \begin{array}{l}0\ \text{if}\ (x,y)\in R_1\ \text{and}\ (x,g(y))\in R_2\dots\\1\ \text{otherwise}\end{array} \right\}\right\} \nonumber\\
  &= \max_{x\in X} \left\{ \#\{Y_2\!\setminus\! g(Y_1)\}\! +\!\! \sum_{y\in Y_1}\! \left\{ \begin{array}{l}0\ \text{if}\ y\in Y_1^{\sigma},\ g(y)\in Y_2^{\sigma},\ \text{and}\ \sigma\owns x\\1\ \text{otherwise}\end{array} \right\}\right\} \label{line1}\\
  &= \max_{x\in X} \left\{ \#\{Y_2\!\setminus\! g(Y_1)\}\! +\!\! \sum_{y\in Y_1}\! \left\{ \begin{array}{l}1\ \text{if}\ y\in Y_1^{\sigma},\ g(y)\notin Y_2^{\sigma},\ \text{and}\ \sigma\owns x\\0\ \text{otherwise}\end{array} \right\}\right\} \label{line2}\\
  &= \max_{x\in X} \left\{ \#\{Y_2\!\setminus\! g(Y_1)\}\! + \!\! \sum_{\sigma\owns x} \# \left\{y\in Y_1^{\sigma} \colon g(y)\notin Y_2^{\sigma}\right\}\right\} \label{line3}\\
  &= \underbrace{\#\{Y_2\!\setminus\! g(Y_1)\}}_{\text{(A)}} + \underbrace{\max_{x\in X} \left\{ \sum_{\sigma\owns x} \#\left\{y\in Y_1^{\sigma} \colon g(y)\notin Y_2^{\sigma}\right\}\right\}}_{\text{(B)}}  \label{line4}
\end{align}

Because of the complicated notation, we elaborate on each line after the initial one, which merely restates definition \ref{df:weight}. Line \eqref{line1} converts the original statement of the sum from relations $R_1, R_2$ to agreeing partitions $Y_1^{\sigma}, Y_2^{\sigma}$. They are equivalent, because the condition in the original statement that for a particular $y\in Y_1$
\begin{equation*}
  ( (x,y)\in R_1\ \text{and}\ (x, g(y))\in R_2)\\\phantom{0}\ \text{or}\ \ ((x,y)\notin R_1\ \text{and}\ (x,g(y)\notin R_2))
\end{equation*}
means that $y$ is in relation to a specific $x\in X$ if and only if $g(y)$ is also in relation to the same $x$. Hence the collections of $x$ with which $y$ and $g(y)$ are in relation are the same
\begin{equation*}
  \sigma_1 =\{x\colon (x,y)\in R_1\}= \{x\colon (x,g(y))\in R_2\}= \sigma_2
\end{equation*}
which means
\begin{equation*}
  \pi(y) = \pi(g(y)).
\end{equation*}
Since $y$ and $g(y)$ must be in partitions defined by the same index, for any given $x\in X$, the original condition is true if and only if
\begin{equation*}
  y\in Y_1^{\sigma}\ \text{and}\ g(y)\in Y_2^{\sigma}
\end{equation*}
for some $\sigma$ that contains $x$. Line \eqref{line2} restates line \eqref{line1} in terms of its logical complement by requiring that $y$ and $g(y)$ be in partitions with different indices. The effect of the sum is to count all $y\in Y_1$ that satisfy this condition. Line \eqref{line3} restates the sum as running over all indices containing $x$ and adding the count of such $y$ for which $g(y)$ is not in an agreeing partition. Finally, line \eqref{line4} merely factors the first term, (A), out of the maximum since (A) only counts the elements in the target's field $Y_2$ that are not found in the function's range $g(Y_2)$, which is independent of $x$.

Term (B) provides the maximum count of $y_1\in Y_1$ mapped to disagreeing partitions of $Y_2$ whose indices include a specific $x$. In terms of this maximizing $\hat{x}_2$, we can write
\begin{equation*}
  \text{(B)}=\sum_{\sigma\owns \hat{x}_2} \#\left\{ y\in Y_1^{\sigma} \colon g(y)\notin Y_2^{\sigma}\right\},
\end{equation*}
taking care to remember that $\hat{x}_2$ depends on both $g$ and the partitioning of $Y_2$.

\label{proofstrategy} Finally, we make some general observations about calculating the distance metric using partitions. Distance, of course, is the larger of the global minimum weights achievable by any functions in one direction compared to that of those in the other direction:
\begin{equation*}
  d(r_1, r_2) = \max \{\min_{g} w(g|r_1,r_2), \min_{g'} w(g'|r_2,r_1)\}.
\end{equation*}

Clearly, the agreeing and disagreeing partitions are fixed by the two relations and do not depend on $g$, but the global minimum of the sum of (A) and (B) depends on the interaction, via $g$, between (A) and (B). In general, (A) is minimized by a mapping $Y_1\to Y_2$ that is as close to surjective as possible, while (B) is minimized by preferentially mapping $y_1\in Y_1$ to agreeing partitions if possible and, when not, to disagreeing partitions of $Y_2$ whose indices do not contain the maximizing $\hat{x}$. Therefore, a one-to-one mapping to agreeing partitions or to disagreeing partitions with indices that do not contain $\hat{x}$ reduces (B) one-for-one without changing (A). Of course, any many-to-one mapping increases (A), regardless of its impact on (B).

These observations lead to a general strategy for demonstrating the maximins required to calculate distance. We will start with a function or class of functions that minimize (A) or (B) in an exact weight calculation with an upper bound, iteratively extend or adjust the function(s) to a broader class if necessary to achieve a global minimum, and then demonstrate that the final weight so derived is a global minimum by comparing the impact on (A) and (B) of all alternative mappings.

\subsection{\texorpdfstring{Bounding Weight of Disagreements: $\kappa$ Definition}{Bounding Weight of Disagreements: \textbackslash kappa Definition}}

Although (B) depends on $g$, we can find an upper bound for it over all functions $g$. In this section we construct such a bound in terms of a function $\kappa\colon r_1\times r_2 \to \N$ that reduces the count of potential mappings to disagreeing sets (the sum in (B)). The construction here assumes that the potential mappings are just for elements of $Y_1$ that are not mapped to agreeing partitions of $Y_2$, i.e., $y\in Y_1^{\not\pis}$, because, as noted in the previous section, (B) will generally be minimized by such a mapping. Of course, $\text{(B)} = 0$ and the bound is unnecessary if there are no disagreeing partitions in either $r_1,r_2$; so we also assume that $\#Y_1^{\not\pis}\neq0$ and $\#Y_2^{\not\pis}\neq0$.

First, recall that the indices classifying any $\relX$ object $r$ are subsets of the indexing set $\pi(Y)\subseteq X$. We can therefore group these indices---and hence group the partitions defined by them---according to whether the indices overlap with respect to one or more $x\in \pi(Y)$.

\begin{proposition}
  \label{proposition:xgroupingofpartitions}
  The $Y$-partitions of any $\relX$ object $r=(Y,R)$ can be divided into disjoint groups $Y_{\tau_i}=\{Y^{\sigma}\colon \sigma\cap\tau_i\neq\emptyset\}$ defined by a collection $T = \{\tau_1,\dots,\tau_N\}$ of disjoint subsets $\tau_i\subseteq \pi(Y)$, where any partitions $Y^{\emptyset}$ are treated as a distinct group $Y_{\tau = \emptyset}$. The collection $\{Y_{\tau_1},\dots, Y_{\tau_N}\}$ of such groups is called the \emph{x-grouping of partitions} of $r$.
\end{proposition}

\begin{proof}
  \label{proof:xgroupingofpartitions}
  We construct the collection of disjoint groups $Y_{\tau_i}$ as follows. First, let $Y_{\emptyset} = \{Y^{\emptyset}\}$. If $Y_{\emptyset}=Y$ then there is only group and we are done. Otherwise $\pi(Y)\neq\emptyset$ and we can choose an initial element $x_1\in\sigma$ for any nonempty $\sigma\in\pi(Y)$, let
  \begin{equation*}
    \tau_1=\bigcup_{\sigma\in\pi(Y)} \{\sigma\owns x_1\}\quad\text{and}\quad Y_{\tau_1}=\{Y^{\sigma}\colon \sigma\in\tau_1\}\quad\text{and}\quad U_1=\{\tau_1\}.
  \end{equation*}
  If $Y=Y_{\emptyset}\cup Y_{\tau_1}$ then we are done. Otherwise $\pi(Y)\setminus U_1\neq\emptyset$ and we can repeat the construction for a $\tau_2$ with some $x_2\in\sigma\in \pi(Y)\setminus U_1$ to collect $Y_{\tau_2}=\{Y^{\sigma}\colon \sigma\in\tau_2\}$ and $U_2=U_1\cup\{\tau_2\}$. By construction, $\tau_1\cap\tau_2=\emptyset$ and $U_1\subset U_2\subseteq\pi(Y)$. Since $r$ is finite, so is $\pi(Y)$ and the iterative construction must terminate with some collection $T = \{\tau_1,\dots,\tau_N\}$ of disjoint subsets of $\pi(Y)$. Since the elements of $\pi(Y)$ constitute the indexing set of $r$, each $\tau_i\in T$ specifies a group of $Y$-partitions $Y_{\tau_i}=\{Y^{\sigma}\colon \sigma\cap\tau\neq\emptyset\ \text{only if}\ \tau=\tau_i\}$. Because the $Y$-partitions $Y^{\sigma}$ are disjoint, so are these groups $Y_{\tau_i}$. We have constructed the disjoint groups of $Y$-partitions as required.
\end{proof}

\begin{example}
  \label{eg:eg_xgrouping}
  Suppose we have a $\relX$ object $(Y,R)$ with $X=\{a,b,c,d\}$ and $Y = \{1, \dots, 10\}$ and the following relation
  \begin{equation*}
    R = \begin{pmatrix} 0 & 1 & 1 & 1 & 0 & 0 & 0 & 0 & 0 & 0\\ 0 & 0 & 1 & 1 & 0 & 0 & 0 & 0 & 0 & 0\\ 0 & 0 & 0 & 0 & 1 & 1 & 0 & 0 & 0 & 0\\ 0 & 0 & 0 & 0 & 0 & 0 & 1 & 1 & 1 & 1 \end{pmatrix}
  \end{equation*}
  Then $\pi(Y)=\{\emptyset, [a], [c], [d], [a,b]\}$. First we note that we have $Y_{\emptyset}=\{Y^{\emptyset}\}\neq Y$ and $U_0=\{\emptyset\}$. Continuing with $[a]$ we have
  \begin{equation*}
    \tau_1 = \bigcup_{\sigma\in\pi(Y)\setminus U_0} \{\sigma\owns [a]\}=\{[a], [a,b]\}\ \text{and}\ Y_{1,\tau_1}=\{Y^{[a]}, Y^{[a,b]}\},\ U_1=\{\emptyset, [a], [a,b]\}.
  \end{equation*}

  Iterating, we also find
  \begin{align*}
    \tau_2 &= \bigcup_{\sigma\in\pi(Y)\setminus U_1} \{\sigma\owns [c]\} = \{[c]\}\ \text{and}\ Y_{1,\tau_2}=\{Y^{[c]}\},\ U_2 = \{\emptyset, [a], [a,b], [c]\}\\
    \tau_3 &= \bigcup_{\sigma\in\pi(Y)\setminus U_2} \{\sigma\owns [d]\} = \{[d]\}\ \text{and}\ Y_{1,\tau_3}=\{Y^{[d]}\},\ U_3 = \{\emptyset, [a], [a,b], [c], [d]\}.
  \end{align*}
  The iteration terminates when $U_3=\pi(Y)$ with $T=\{\emptyset, \{[a], [a,b]\}, \{[c]\}, \{[d]\}\}$ and the disjoint groups of partitions
  \begin{equation*}
    Y=\{\{Y^{\emptyset}\}, \{Y^{[a]}, Y^{[a,b]}\}, \{Y^{[c]}\}, \{Y^{[d]}\} \}
  \end{equation*}
  which correspond to groups of columns in the binary matrix
  \begin{equation*}
    R = \left(
    \begin{matrix}
    0\\
    0\\
    0\\
    0\\
    \end{matrix}\ \rule[-4.5ex]{1pt}{10.5ex}\ \begin{matrix}
    1 & 1 & 1\\
    0 & 1 & 1\\
    0 & 0 & 0\\
    0 & 0 & 0
    \end{matrix}\ \rule[-4.5ex]{1pt}{10.5ex}\ \begin{matrix}
    0 & 0\\
    0 & 0\\
    1 & 1\\
    0 & 0
    \end{matrix}\ \rule[-4.5ex]{1pt}{10.5ex}\ \begin{matrix}
    0 & 0 & 0 & 0\\
    0 & 0 & 0 & 0\\
    0 & 0 & 0 & 0\\
    1 & 1 & 1 & 1
    \end{matrix}
    \right)
  \end{equation*}
\end{example}

To bound (B), we apply proposition \ref{proposition:xgroupingofpartitions} to $Y_2^{\not\pis}$ to create a particular mapping with a (B) value that can readily be calculated based on the $\relX$ objects alone.

\begin{proposition}
  \label{proposition:kappa_mapping}
  Given two $\relX$ objects $r_1,r_2$, it is always possible to find a function $\phi\colon Y_1^{\not\pis}\to Y_2^{\not\pis}$ that is a one-for-one assignment of up to $\min\{\#Y_1^{\not\pis}, \#Y_2^{\not\pis}\}$ elements $y_1\in Y_1^{\not\pis}$ to elements of partitions of $Y_2^{\not\pis}$ selected from disjoint groups of partitions $T=\{Y_{2,\tau_1}^{\not\pis}, \dots, Y_{2,\tau_N}^{\not\pis}\}$ in increasing order of the total cardinality of the group of partitions. The map $\phi$ is called a $\kappa$ \emph{mapping} from $r_1$ to $r_2$.
\end{proposition}

\begin{proof}
  \label{proof:kappa_mapping}
  The proposition is trivially true when $\min\{\#Y_1^{\not\pis}, \#Y_2^{\not\pis}\}=0$. Otherwise, we apply proposition \ref{proposition:xgroupingofpartitions} to construct the collection $Y_2^{\not\pis} = \cup\{Y_{2,\tau}^{\not\pis}\}$ of disjoint groups of partitions. Let $\#Y_{2,\tau}^{\not\pis} = \sum_{\sigma\in\tau} \#Y_2^{\sigma}$ denote the total cardinality of each group. We can then construct a lexicographic ordering of partitions $Y_2^{\sigma_i}\in Y_{2,\tau_j}^{\not\pis}$ according to the increasing total cardinality of each group and within each group the increasing multiplicity of each partition. Finally define $\phi\colon Y_1^{\not\pis} \to Y_2^{\not\pis}$ by assigning to a first $y_1\in Y_1^{\not\pis}$ any $y'_1$ in the lexicographically first partition $Y_2^{\sigma_1}\in Y_{2,\tau_1}^{\not\pis}$ and then iterating with unique assignments $y_i$ to $y'_i$ in that partition followed by the next partition in lexicographic order, etc., such that $y'_i\notin \phi(\{y_1,\dots,y_{i-1}\})$, until $i=\min\{Y_1^{\not\pis}, Y_2^{\not\pis}\}$ and, if $\#Y_1^{\not\pis} > \#Y_2^{\not\pis}$, assigning any remaining $y\in Y_1^{\not\pis}$ to the same $y'_1$. That completes the required construction.
  \end{proof}

\begin{example}
  \label{eg:eg_kappa_mapping}
  Suppose now that we want to construct a $\kappa$ mapping to the object $r_2=(Y,R)$ of the previous example (\ref{eg:eg_xgrouping}) from $r_1=(Y',R')$ where $Y' = \{11, \dots, 15\}$ and
  \begin{equation*}
    R' = \begin{pmatrix} 0 & 1 & 1 & 1 & 1\\ 1 & 0 & 1 & 0 & 1\\ 0 & 1 & 1 & 0 & 0\\ 0 & 0 & 0 & 1 & 1\\ \end{pmatrix}
  \end{equation*}
  Conveniently, we already had sorted and grouped the target $R$:
  \begin{equation*}
    R = \left(
    \begin{matrix}
    0\\
    0\\
    0\\
    0\\
    \end{matrix}\ \rule[-4.5ex]{1pt}{10.5ex}\ \begin{matrix}
    1 & 1 & 1\\
    0 & 1 & 1\\
    0 & 0 & 0\\
    0 & 0 & 0
    \end{matrix}\ \rule[-4.5ex]{1pt}{10.5ex}\ \begin{matrix}
    0 & 0\\
    0 & 0\\
    1 & 1\\
    0 & 0
    \end{matrix}\ \rule[-4.5ex]{1pt}{10.5ex}\ \begin{matrix}
    0 & 0 & 0 & 0\\
    0 & 0 & 0 & 0\\
    0 & 0 & 0 & 0\\
    1 & 1 & 1 & 1
    \end{matrix}
    \right)
  \end{equation*}
  Comparing columns, we readily observe that no columns in $R'$ are found in $R$, i.e., all partitions of $Y$ disagree with all partitions of $Y'$. Using multiset notation showing multiplicities, the lexicographic ordering of partitions is
  \begin{equation*}
    Y^{\emptyset}\!:\!1,\ Y^{[c]}\!:\!2,\ Y^{[a]}\!:\!1,\ Y^{[a,b]}\!:\!2,\ Y^{[d]}\!:\!4.
  \end{equation*}

  One corresponding order of elements of $Y$ would be
  \begin{equation*}
    1,5,6,2,3,4,7,8,9,10
  \end{equation*}
  Then a $\kappa$ mapping $\phi\colon Y'\to Y$ would be
  \begin{equation*}
    \phi(11)=1,\ \phi(12)=5,\ \phi(13)=6,\ \phi(14)=2,\ \phi(15)=3
  \end{equation*}
  which is illustrated in figure \ref{eg_kappa_mapping}.
\end{example}

\begin{figure}[ht]
  \includegraphics[height=1.5in]{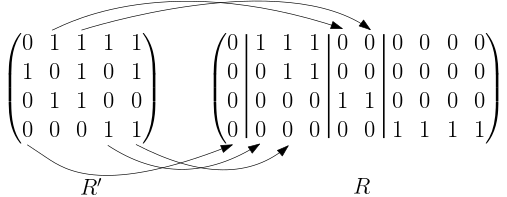}
  \caption{Relations $R'$ and $R$ with $\kappa$ mapping for Example \ref{eg:eg_kappa_mapping}.}
  \label{eg_kappa_mapping}
\end{figure}

Since a $\kappa$-mapping is always possible, a weight-minimizing $g$ can have no (B) larger than the value achieved by it. This value will be our upper bound. To calculate it we now define the following function.

\begin{definition}
  \label{df:kappa_function}
  Given two $\relX$ objects $r_1,r_2$, let $T_n=\{Y_{2,\tau_1}^{\not\pis}, \dots, Y_{2,\tau_N}^{\not\pis}\}$ be the sequence of $x$-grouped partitions of $Y_2^{\not\pis}$ in increasing order of their total cardinality, $S_n=\{s_n\}_{i=1}^{N}$ the sequence of partial sums $s_n=\sum_{i=1}^{n} \#Y_{2,\tau_i}^{\not\pis}$, and $m\le N$ the largest index such that $s_m \le \#Y_{1}^{\not\pis}$. Define the function $\kappa\colon Obj(\relX)\times Obj(\relX) \to \N$ by
  \begin{align*}
    \kappa(r_1, r_2) &:= \begin{cases}
    0 &\text{if}\ N = 1\\
    s_{m-1} &\text{if}\ \#Y_1^{\not\pis} \ge \#Y_2^{\not\pis}\\
    s_{m-1} &\text{if}\ \#Y_1^{\not\pis} < \#Y_2^{\not\pis}\ \text{and}\ \#Y_{2,\tau_m}^{\not\pis} > \#Y_{1}^{\not\pis} - s_{m}\\
    s_m &\text{otherwise}
    \end{cases}
  \end{align*}
\end{definition}

\begin{example}
  \label{eg:eg_kappa_function}
  Continuing the previous examples \ref{eg:eg_xgrouping} and \ref{eg:eg_kappa_mapping},
  \begin{equation*}
    T_n = \{Y^{\emptyset}, Y^{[c]}, Y^{[a], [a,b]}, Y^{[d]}\}
  \end{equation*}
 and
  \begin{equation*}
    S_n = \{1, 3, 6, 10\}.
  \end{equation*}
  Since $\#Y'^{\not\pis}=5$, we have $Y_{\tau_m}=Y^{[c]}$, so $m=2$, $\#Y'^{\not\pis} = 5 < 10 = \#Y^{\not\pis}$, but $\#Y_{\tau_m}^{\not\pis}=3 = 5-2 = \#Y'^{\not\pis}-s_m$. Hence $\kappa(r_1, r_2) = s_{m} = 3$.
\end{example}

We are now able to state and prove the $\kappa$ bound.

\begin{lemma}
  \label{lemma:kappabound}
  Given $\relX$ objects $r_1, r_2$, the count of $y\in Y_1^{\not\pis}$ mapped to disagreeing partitions in $Y_2^{\not\pis}$ by any function $g\colon Y_1\to Y_2$ minimizing the weight $w(g|r_1,r_2)$ cannot exceed the total multiplicity of the disagreeing partitions $Y_1^{\not\pis}$ reduced by the quantity $\kappa(r_1, r_2)$, i.e.,
  \begin{equation*}
    (B) = \max_{x\in X} \left\{ \sum_{\sigma\owns x} \#\left\{y\in Y_1^{\sigma} \colon g(y)\notin Y_2^{\sigma}\right\}\right\} \le \#Y_1^{\not\pis} - \kappa(r_1, r_2).
  \end{equation*}
\end{lemma}

\begin{proof}
  \label{proof:kappabound}

  As mentioned in the previous section, since (B) counts only those $y\in Y_1^{\not\pis}$ mapped to the group of disagreeing partitions of $Y_2^{\not\pis}$ collectively receiving the most mappings of $y\in Y_1^{\not\pis}$ and all including the same maximizing $\hat{x}_2$, (B) can be reduced by the count of any mappings to partitions in other $x$-groups of disagreeing partitions. By construction, a $\kappa$ mapping maps one-to-one and surjectively to disagreeing partitions in lexicographic order such that the last $x$-grouping of partitions to which mappings are made will have the potential to receive the largest count of such mappings. There are four cases:
  \begin{enumerate}
    \def\labelenumi{(\arabic{enumi})}
    \item
      If there is only one $x$-grouping of partitions, i.e., when $N=1$,
      the weight-maximizing $\hat{x}_2$ must be in that group, and no
      reductions are possible; hence $\kappa(r_1, r_2) = 0$.
    \item
      If $N>1$ and the count of $y$-elements in the source's disagreeing
      partitions equals or exceeds that of the target's disagreeing
      partitions, i.e., $\#Y_1^{\not\pis} \ge \#Y_2^{\not\pis}$, then the
      weight-maximizing $\hat{x}_2$ must be in the last $x$-grouping,
      $Y_{2,\tau_m}$, and hence $\kappa(r_1,r_2) = s_{m-1}$.
    \item
      If $N>1$ and $\#Y_1^{\not\pis} < \#Y_2^{\not\pis}$, but the count
      to the last surjectively mapped $x$-grouping of partitions exceeds
      the count mapped to the next $x$-group, i.e.,
      $\#Y_{2,\tau_m} > \#Y_{2,\tau_{m+1}}$, then the weight-maximizing
      $\hat{x}_2$ is in $Y_{2,\tau_m}^{\not\pis}$ and again
      $\kappa(r_1, r_2) = s_{m-1}$.
    \item
      Otherwise $N>1,\ \#Y_1^{\not\pis} < \#Y_2^{\not\pis}$, and
      $\#Y_{2,\tau_m} \le \#Y_{2,\tau_{m+1}}$, i.e., the count of mappings
      to the $x$-group $Y_{2,\tau_{m+1}}^{\not\pis}$ equals or exceeds
      the count of mappings to each of the other $x$-groups
      $\{Y_{2,\tau_{1}}^{\not\pis}, \dots, Y_{2,\tau_{m}}^{\not\pis}\}$
      and $\kappa(r_1,r_2)=s_m$.
  \end{enumerate}

  By proposition \ref{proposition:kappa_mapping}, a $\kappa$ mapping $\phi\colon r_1\to r_2$ always exists. Consequently, a weight-minimizing $g$ can have no (B) larger than the value achieved by it. Thus a weight-minimizing $g$ must have
    \begin{equation*}
      \text{(B)} \le \#Y_1^{\not\pis} - \kappa(r_1, r_2). \qedhere
    \end{equation*}
\end{proof}

We can illustrate the four cases in the proof of the $\kappa$ bound by considering variations on examples \ref{eg:eg_xgrouping}, \ref{eg:eg_kappa_mapping}, and \ref{eg:eg_kappa_function}.

\begin{example}
  \label{eg:eg_kappa_thrm_case1}
  Suppose we have $r_1=(Y', R')$ from example \ref{eg:eg_kappa_mapping}, for which
  \begin{equation*}
    R' = \begin{pmatrix} 0 & 1 & 1 & 1 & 1\\ 1 & 0 & 1 & 0 & 1\\ 0 & 1 & 1 & 0 & 0\\ 0 & 0 & 0 & 1 & 1\\ \end{pmatrix}
  \end{equation*}
  and we wish to calculate the $\kappa$ bound on mappings to $r_3=(Y_3,R_3)$ where $Y_3 = \{2,3,4\}$ and
  \begin{equation*}
    R_3 = \left(
      \begin{matrix} 1 & 1 & 1\\
      0 & 1 & 1\\
      0 & 0 & 0\\
      0 & 0 & 0
      \end{matrix} \right)
  \end{equation*}
  The analysis in example \ref{eg:eg_xgrouping} showed that $R_3$ has only one $x$-grouping of partitions, namely $\{Y_3^{[a], [a,b]}\}$. Accordingly, when calculating the $\kappa$ bound we find $N=1$. The $\kappa$ mapping must map all $y\in Y'$ to this group of partitions. No reductions are possible, and $\kappa(r_1,r_3) = 0$.
\end{example}

\begin{example}
  \label{eg:eg_kappa_thrm_case2}
  Now suppose we still have $r_1=(Y', R')$ from example \ref{eg:eg_kappa_mapping} and we wish to calculate the $\kappa$ bound on mappings to $r_4=(Y_4, R_4)$ with $Y_4=\{1,2,3,4\}$ and
  \begin{equation*}
    R_4 = \left( \begin{matrix}
    0\\
    0\\
    0\\
    0\\
    \end{matrix}\ \rule[-4.5ex]{1pt}{10.5ex}\ \begin{matrix}
    1 & 1 & 1\\
    0 & 1 & 1\\
    0 & 0 & 0\\
    0 & 0 & 0
    \end{matrix} \right)
  \end{equation*}
  where the vertical bar reflects the $x$-grouping of partitions as in example \ref{eg:eg_xgrouping}. Now $N=2$ and the count of disagreeing columns in source $r_1$ exceeds the count of disagreeing columns in target $r_4$. A $\kappa$ mapping $\phi\colon Y'\to Y_4$ would be
  \begin{equation*}
    \phi(11)=1,\ \phi(12)=2,\ \phi(13)=3,\ \phi(14)=4,\ \phi(15)=4.
  \end{equation*}
  So in calculating the $\kappa$ function we find $m=2$ and $\kappa(r_1, r_4) = 1$, reflecting the reduction by the mapping $\phi(11)=1$.
\end{example}

\begin{example}
  \label{eg:eg_kappa_thrm_case3}
  For an example of the third case, suppose we have want to calculate the $\kappa$ bound on mappings from $r_5=(Y_5, R_5)$ with just $Y_5=\{11, 12, 13, 14\}$ and
  \begin{equation*}
    R_5 = \begin{pmatrix}
    0 & 1 & 1 & 1\\
    1 & 0 & 1 & 0\\
    0 & 1 & 1 & 0\\
    0 & 0 & 0 & 1
    \end{pmatrix}
  \end{equation*}
  to the full target $r=(Y,R)$ from example \ref{eg:eg_xgrouping} for which
  \begin{equation*}
    R = \left( \begin{matrix}
    0\\
    0\\
    0\\
    0\\
    \end{matrix}\ \rule[-4.5ex]{1pt}{10.5ex}\ \begin{matrix}
    1 & 1 & 1\\
    0 & 1 & 1\\
    0 & 0 & 0\\
    0 & 0 & 0
    \end{matrix}\ \rule[-4.5ex]{1pt}{10.5ex}\ \begin{matrix}
    0 & 0\\
    0 & 0\\
    1 & 1\\
    0 & 0
    \end{matrix}\ \rule[-4.5ex]{1pt}{10.5ex}\ \begin{matrix}
    0 & 0 & 0 & 0\\
    0 & 0 & 0 & 0\\
    0 & 0 & 0 & 0\\
    1 & 1 & 1 & 1
    \end{matrix} \right)
  \end{equation*}
  with vertical bars again reflecting $x$-groupings in that example. Now $N=4$ but the count of disagreeing columns in the source is less than the count of disagreeing columns in the target. A $\kappa$ mapping $\phi\colon Y_5\to Y$ would be
  \begin{equation*}
    \phi(11)=1,\ \phi(12)=5,\ \phi(13)=6,\ \phi(14)=2.
  \end{equation*}

  Calculating the $\kappa$ function we find $m=2$ and the count mapped to the last surjectively-mapped $x$-group $\#Y_{\{[c]\}}=2$, which exceeds the count of 1 remaining mapping to $Y_{\{[a], [a,b]\}}$. So $\kappa(r_5, r) = 1$, reflecting the reduction by the mapping $\phi(11)=1$.
\end{example}

\begin{example}
  \label{eg:eg_kappa_thrm_case4}
  Finally, example \ref{eg:eg_kappa_mapping} illustrates the fourth case. With the $\kappa$ mapping given in that example, we have $N=4$ and, as calculated in example \ref{eg:eg_kappa_function}, $m=2$ but the count of 2 mapped to the last surjectively-mapped $x$-group $Y_{\{[c]\}}$ equals the count of 2 mapped to the next $x$-group $Y_{\{[a], [a,b]\}}$ and hence $\kappa(r_1,r_4) = 3$, reflecting reduction by the mappings
  \begin{equation*}
    \phi(11)=1,\ \phi(12)=5,\ \phi(13)=6.
  \end{equation*}
\end{example}

Thus $\kappa(r_1,r_2)$ is a minimum number of disagreeing mappings by \emph{any} function from $Y_1$ into partitions of $Y_2$ that do not have any relation to the maximizing $\hat{x}$ and so avoid being counted in term (B) of the weight function. In effect, $\kappa(r_1,r_2)$ caps the penalty of failing to have morphisms between $r_1$ and $r_2$. Note that this quantity may only rarely be the \emph{minimum} upper bound. For instance, some function $g$ might map less than surjectively to partitions in additional $x$-groups without affecting the maximizing $\hat{x}$; those additional mappings would be ignored by $\kappa$ and further reduce the upper bound.

\subsection{\texorpdfstring{$\kappa$ Algorithm}{\textbackslash kappa Algorithm}}
\label{section:algorithm}

Before continuing the discussion of weight and distance, we demonstrate the practicality of $\kappa$. When assuming all partitions disagree with those in the source relation,\footnote{The algorithm outlined here assumes a ``precalculation'' or ``removal'' of disagreeing partitions in $R1,R2$. Incorporating this initial step would obscure the core algorithm being highlighted here. An efficient combination of both steps folds them together, further obscuring the core $\kappa$ algorithm.} the $\kappa$ calculation can be described abstractly as follows:

\renewcommand{\algorithmicrequire}{\textbf{IN:}\quad\ }
\renewcommand{\algorithmicensure}{\textbf{OUT:}}
\begin{algorithm}
  \caption{$\kappa$ Algorithm}
  \label{kappa_alg}
  \begin{algorithmic}[1]
    \REQUIRE $R$ $<$binary matrix of a relation$>$\\
      $\quad\ \ MAX$ $<$integer count of source's disagreeing partitions$>$
    \ENSURE $KAPPA$ $<$integer of target's $\kappa$ value$>$
    \STATE $PARTITIONS \leftarrow$ subsets of $R$ columns with same rows
    \STATE $XGROUPINGS \leftarrow$ join elements of $PARTITIONS$ whose columns pairwise share an $X$
    \STATE $BLOCKCOUNTS \leftarrow$ count columns each element of $XGROUPINGS$
    \STATE Sort $BLOCKCOUNTS$ in increasing order
    \STATE $BLOCKSUMS \leftarrow$ cumulative sums of $BLOCKCOUNTS$
    \STATE Sort $BLOCKSUMS$ in increasing order
    \STATE $M \leftarrow$ index of largest element $\le MAX$ in $BLOCKSUMS$
    \IF{$length(BLOCKSUMS) = 1$}
      \STATE $KAPPA = 0$
    \ELSIF{$MAX \ge$ number of columns in $R$}
      \STATE $KAPPA = BLOCKSUMS[M-1]$
    \ELSIF{$BLOCKSUMS[M] > BLOCKSUMS[M+1]$}
      \STATE $KAPPA = BLOCKSUMS[M-1]$
    \ELSE
      \STATE $KAPPA = BLOCKSUMS[M]$
    \ENDIF
    \RETURN $KAPPA$.
  \end{algorithmic}
\end{algorithm}

This algorithm takes as input a binary matrix representing a target relation and an integer representing the maximum number of mappings from the source relation to $y$-partitions of the target relation (i.e., to columns of the input matrix) and then outputs the $\kappa$ value.  Source code can be obtained from \url{https://github.com/kpewing/relations}.  This Python 3 library includes unit tests based on examples \ref{eg:eg_kappa_thrm_case1} through \ref{eg:eg_kappa_thrm_case4} (as well as additional code to complete the calculation of distance between two relations).

For a relation of dimension $m \times n$, the time-limiting step is the first---constructing a list of $x$-groupings---during which the algorithm iterates over all columns of the relation to check each row against the list of $x$'s grouped so far. This step terminates because the rows and columns are finite and has duration at worst proportional to $\mathcal{O}(m\times n)$, and memory proportional to the number of rows: $\mathcal{O}(m)$. The second step---sorting the cumulative sums---can be achieved in time and memory at worst proportional to $\mathcal{O}(m + n)$ using an integer sorting algorithm like \texttt{bucket} or \texttt{counting} (see \cite[]{wikipedia_sorting_nodate} (comparing sorting algorithms)). The next two steps each require a single traversal of the list of $x$-groupings, which requires no additional memory and time at worst proportional to the number of rows: $\mathcal{O}(m)$. The remaining steps are trivial integer calculations and comparisons.

\subsection{Weight Calculation and Bounds}

We now turn to the weight of functions $g\colon r_1\to r_2$ between two $\relX$ objects.

Recall that if $r_1=(Y_1,R_1)$ and $r_2=(Y_2,R_2)$, a morphism $r_1 \to r_2$ in ${\bf Rel}_X$ simply consists of a function $g: Y_1 \to Y_2$ such that $(x,g(y))\in R_2$ whenever $(x,y) \in R_1$.  Therefore, the weight of $g$
\begin{equation*}
  w(g | r_1, r_2) = \max_{x \in X}\left\{ \#\{Y_2 \backslash g(Y_1)\} + \sum_{y\in Y_1} \left\{\begin{aligned}0 & \text{ if }((x,y)\in R_1 \text{ and } (x,g(y)) \in R_2)\\& \text{ or } ((x,y) \notin R_1 \text{ and } (x,g(y)) \notin R_2)\\ 1 &\text{ otherwise}\end{aligned}\right\}\right\}
\end{equation*}
is well-defined.  Because $g$ and $g'$ range over all functions in the definition of the pseudometric $d$, these ranges include all morphisms as well.  Therefore, the following Corollary is immediate.

\begin{corollary}
  \label{cor:weight1}
  If $r_1=(Y_1,R_1)$ and $r_2 = (Y_2,R_2)$ are objects in ${\bf Rel}_X$ then at least one of the following are true:
  \begin{enumerate}
  \item $d((Y_1,R_1),(Y_2,R_2)) \le w(g | r_1, r_2)$ for every morphism $g: (Y_1,R_1) \to (Y_2,R_2)$, or
  \item $d((Y_1,R_1),(Y_2,R_2)) \le w(h | r_2, r_1)$ for every morphism $h: (Y_2,R_2) \to (Y_1,R_1)$.
  \end{enumerate}
\end{corollary}

The order of the quantifiers is significant, since the order is a result of the particular structure of the minimum and maximum operators in the definition of the pseudometric $d$.

\begin{example}
  \label{eg:eg_morphism_weight}
  Recall the relations $r_1$ and $r_2$ given in Example \ref{eg:eg_total_weight}, with matrices given by
  \begin{equation*}
    R_1 = \begin{pmatrix}
      1&1&1&1&0\\
      1&1&0&0&0\\
      0&0&1&1&1\\
      0&0&1&1&1\\
    \end{pmatrix}, \;
    R_2 =\begin{pmatrix}
    1&1&1&1&1\\
    1&1&0&0&0\\
    0&0&1&1&1\\
    0&0&0&1&1\\
    \end{pmatrix}.
  \end{equation*}
  In Example \ref{eg:eg_total_weight}, we determined the distance between these two relations was $1$.  If we label the rows of each relation as $Y=\{1,2,3,4,5\}$, then it is easy to verify that the function $f: (Y,R_1) \to (Y,R_2)$ given by
  \begin{equation*}
    f(1) = 1,\;     f(2) = 1,\;     f(3) = 4,\;     f(4) = 4,\;     f(5) = 4,
  \end{equation*}
  is a morphism.  Corollary \ref{cor:weight1} asserts that the weight of this morphism is at least the distance between $R_1$ and $R_2$.  This is easily verified by direct calculation
  \begin{eqnarray*}
    w(f|r_1, r_2) &=& \max_{x\in X} \left\{ \#Y_2\!\setminus\! f(Y) +\! \sum_{y\in Y}\! \left\{\begin{array}{l}0\ \text{if}\ ( (x,y)\in R_1\ \text{and}\ (x, f(y))\in R_2))\\\phantom{0}\ \text{or}\ ((x,y)\notin R_1\ \text{and}\ (x,f(y)\notin R_2))\\1\ \text{otherwise}\end{array} \!\!\right\} \!\right\}\\
    &=& 3 + 1 = 4 \ge 1 = d(r_1,r_2).
  \end{eqnarray*}
\end{example}

The weight of a composition of two functions is subadditive in an interesting way.

\begin{corollary}
  \label{cor:weight2}
  Suppose that $r_1=(Y_1,R_1)$, $r_2=(Y_2,R_2)$, and $r_3=(Y_3,R_3)$ are objects in ${\bf Rel}_X$.
  Recounting the proof of the triangle inequality in Proposition \ref{prop:disapseudometric} establishes that
  \begin{equation*}
    w(g \circ f | r_1, r_3) \le w(f| r_1, r_2) + w(g | r_2, r_3)
  \end{equation*}
  whenever $f:(Y_1,R_1) \to (Y_2,R_2)$ and $g:(Y_2,R_2) \to (Y_3,R_3)$ are morphisms in ${\bf Rel}_X$.
\end{corollary}

\begin{example}
  Let us continue Example \ref{eg:eg_morphism_weight}, and recall its definition of the morphism $f: (Y,R_1) \to (Y,R_2)$.  Consider the relation $r_3$ given by the matrix
  \begin{equation*}
    R_3 = \begin{pmatrix}
      1&1\\
      1&0\\
      1&1\\
      0&1\\
      \end{pmatrix},
  \end{equation*}
  where we will choose to label the columns as $Y_3 = \{1,2\}$.  If we define another function $g:(Y,R_2) \to (Y_3,R_3)$ by
  \begin{equation*}
    g(1) = 1,\;     g(2) = 1,\;     g(3) = 2,\;     g(4) = 2,\;     g(5) = 2,
  \end{equation*}
  it is easy to see that $g$ is a morphism.

  Since composition of morphisms in $\relX$ yields morphisms, it is the case that
  \begin{equation*}
    (g \circ f)(1) = 1,\;     (g \circ f)(2) = 1,\;    (g \circ f)(3) = 2,\;    (g \circ f)(4) = 2,\;    (g \circ f)(5) = 2\;
  \end{equation*}
  is a morphism $(g \circ f) : (Y,R_1) \to (Y_3,R_3)$.

  With these facts in hand we can directly verify Corollary \ref{cor:weight2}.
  The weight of the composition is
  \begin{eqnarray*}
    w(g\circ f| r_1,r_3) &=& \max_{x\in X} \left\{ \#Y_3\!\setminus\! (g\circ f)(Y) +\! \sum_{y\in Y}\! \left\{\begin{array}{l}0\ \text{if}\ ( (x,y)\in R_1\ \text{and}\ (x, (g\circ f)(y))\in R_3))\\\phantom{0}\ \text{or}\ ((x,y)\notin R_1\ \text{and}\ (x,(g\circ f)(y)\notin R_3))\\1\ \text{otherwise}\end{array} \!\!\right\} \!\right\}\\
    &=& 0 + 2 = 2.
  \end{eqnarray*}
  We already computed
  \begin{equation*}
    w(f| r_1,r_2) = 4
  \end{equation*}
  in Example \ref{eg:eg_morphism_weight}.  Finally,
  \begin{eqnarray*}
    w(g| r_2,r_3) &=& \max_{x\in X} \left\{ \#Y_3\!\setminus\! g(Y) +\! \sum_{y\in Y}\! \left\{\begin{array}{l}0\ \text{if}\ ( (x,y)\in R_2\ \text{and}\ (x, g(y))\in R_3))\\\phantom{0}\ \text{or}\ ((x,y)\notin R_2\ \text{and}\ (x,g(y)\notin R_3))\\1\ \text{otherwise}\end{array} \!\!\right\} \!\right\}\\
    &=& 0 + 2 = 2.
  \end{eqnarray*}
  This is indeed in agreement with Corollary \ref{cor:weight2} since $2 \le 4 + 2 = 6$.
\end{example}

As the following two propositions demonstrate, the exact weights and bounds thereon depend on whether $g$ is a $\relX$ morphism.

\begin{proposition}
  \label{proposition:morphismweight}
  The minimum weight of any morphism $g:r_1\to r_2$ between two $\relX$ objects is
  \begin{equation*}
    \min_{g} w(g|r_1, r_2) = \max \{\#Y_1, \#Y_2\} - \#Y_1.
  \end{equation*}
\end{proposition}

\begin{proof}
  \label{proof:morphismweight}
  If $g$ is a morphism, then it satisfies the ``whenever'' requirement that $(x,g(y))\in R_2$ whenever $(x,y)\in Y_1$. This means that, given any $y$,
  \begin{equation*}
    \pi(y)=\{x\colon (x,y)\in R_1\}=\{x\colon (x,g(y))\in R_2\}=\pi(g(y)),
  \end{equation*}
  i.e., $g$ honors partitions: $y\in Y_1^{\sigma}$ implies $g(y)\in Y_2^{\sigma}$. Consequently, (B) = 0 for all morphisms. Hence
  \begin{equation*}
    w(g|r_1,r_2) = \#Y_2 - \#g(Y_1)
  \end{equation*}
  This quantity is minimized the more nearly surjective $g$ is, yielding the minimum weight for a morphism
  \begin{equation*}
    \min_{g} w(g|r_1, r_2) = \max \{0, \#Y_2 - \#Y_1\} = \max \{\#Y_1, \#Y_2\} - \#Y_1. \qedhere
  \end{equation*}
\end{proof}

\begin{proposition}
  \label{proposition:nonmorphismweight}
  The minimum weight of any non-morphism $g\colon r_1\to r_2$ between two $\relX$ objects is
  \begin{align*}
    \min_{g} w(g|r_1,r_2) &= \max \{\#Y_1, \#Y_2\} - \#Y_1 + \sum_{\sigma\owns \hat{x}_2} \#\{y\in Y_1^{\sigma}\colon g(y)\notin Y_2^{\sigma}\}\\ &\le \max \{\#Y_1, \#Y_2\} - (\#Y_1^{\pis} + \kappa(r_1,r_2)).
  \end{align*}
\end{proposition}

\begin{proof}
  \label{proof:nonmorphismweight} If $g$ is not a morphism, then it fails to honor partitions for some $y$, i.e., $\pi(y)\neq\pi(g(y))$, which implies that (B) \textgreater{} 0 and like (A) depends on $g$. Consider the simplest case that there are no agreeing partitions and only one disagreeing partition $Y_2^{\sigma\notin\pis}=Y_2$. Then any constant mapping $g \equiv y_2$ has weight
  \begin{equation*}
    w(g|r_1,r_2) = \#Y_2 - 1 + \#Y_1,
  \end{equation*}
  and a non-constant minimizer maps as close to surjectively as possible over the sole partition, yielding as minimum weight:
  \begin{equation*}
    \min_{g} w(g|r_1,r_2) = \#Y_2 - \min\{\#Y_1, \#Y_2\} + \#Y_1 = \max\{\#Y_1, \#Y_2\}. \tag{1}
  \end{equation*}

  If $Y_2$ consists of only multiple disagreeing partitions $Y_2^{\not\pis}$, then (B) depends on the allocation of mappings to them and is maximized by a one-to-one mapping to some collection $Y_2^{\not\pis\owns \hat{x}}$ of disagreeing partitions whose indices all contain the maximizing $\hat{x}$. As discussed in proving lemma \ref{lemma:kappabound}, (B) can be reduced; this, however, must increase (A) one-for-one to the extent that $\#Y_1 > \#Y_2^{\not\pis\not\owns \hat{x}}$.

  Hence, by lemma \ref{lemma:kappabound}, when $Y_2$ consists solely of multiple disagreeing partitions for such a mapping $g$, a minimum weight is
  \begin{align*}
    \min_{g} w(g|r_1,r_2) &= \#Y_2 - \min \{\#Y_1, \#Y_2\} + \sum_{\sigma\owns \hat{x}_2} \#\{y\in Y_1^{\sigma}\colon g(y)\notin Y_2^{\sigma}\} \tag{2}\\
    &\le \#Y_2 - \min \{\#Y_1, \#Y_2\} +  \#Y_1^{\not\pis} - \kappa(r_1,r_2)\\
    &= \#Y_2 - \min \{\#Y_1, \#Y_2\} +  \#Y_1 - (\#Y_1^{\pis} + \kappa(r_1,r_2))\\
    &= \max \{\#Y_1, \#Y_2\} - (\#Y_1^{\pis} + \kappa(r_1,r_2)).
  \end{align*}
  The exact form of line (2) is a global minimum when there are no agreeing partitions, because any other mapping in these circumstances must increase (A) and can at best reduce (B) one-for-one. The inequality form of (2) provides an upper bound on the weight for any global minimizer $g$ with no agreeing partitions. Note that when there is only one, disagreeing partition, both the exact and inexact forms of (2) recover (1).

  Finally, consider the most comprehensive case, with some agreeing partitions along with the disagreeing partitions. If we start with the minimizing mapping for (2) and merely add a mapping as nearly surjective as possible from agreeing to agreeing partition, then the weight will still be calculated by (2).

  It turns out, perhaps surprisingly, that this is a global minimum for non-morphisms. For remapping from agreeing to disagreeing partitions must increase (A) while at best leaving (B) unchanged, and any remapping from a disagreeing to an agreeing partition must also increase (A) while at best reducing (B) one-for-one. This completes the proof.
\end{proof}

Since the summation (B) in the exact form for non-morphisms equals 0 for a morphism, we may combine the two propositions into a general form.

\begin{lemma}
  \label{lemma:minimumweight}
  The minimum weight of any mapping $g\colon r_1\to r_2$ between two $\relX$ objects is
  \begin{align*}
    \min_{g} w(g|r_1,r_2) &= \max \{\#Y_1, \#Y_2\} - \#Y_1 + \sum_{\sigma\owns \hat{x}_2} \#\{y\in Y_1^{\sigma}\colon g(y)\notin Y_2^{\sigma}\}\\ &\le \max \{\#Y_1, \#Y_2\} - (\#Y_1^{\pis} + \kappa(r_1,r_2)).
  \end{align*}
\end{lemma}

The weight of mappings $g'\colon r_2\to r_1$ is derived symmetrically.

\subsection{Distance Calculation and Bounds}

Lemma \ref{lemma:minimumweight} readily allows us to derive a general form for the distance.

\begin{lemma}
  \label{lemma:distance}
  The distance between two $\relX$ objects $r_1, r_2$ is
  \begin{align*}
    d(r_1, r_2) &= \max \{\#Y_1, \#Y_2 \} - \min \left \{\begin{array}{l}
      \#Y_1 - \sum_{\sigma\owns \hat{x}_2} \#\{y\in Y_1^{\sigma}\colon g(y)\notin Y_2^{\sigma} \} \\
      \#Y_2 - \sum_{\sigma\owns \hat{x}_1} \#\{y\in Y_2^{\sigma}\colon g'(y)\notin Y_1^{\sigma} \}
    \end{array} \right\}\\
    &\le \max \{\#Y_1,\#Y_2\} - \min \left\{ \#Y_1^{\pis} + \kappa(r_1,r_2), \#Y_2^{\pis} + \kappa(r_2,r_1) \right\}
  \end{align*}
\end{lemma}

\begin{proof}
  We confirm by calculating:
  \begin{align*}
    d(r_1, r_2) &= \max \{\min_{g} w(g|r_1, r_2), \min_{g'} w(g'|r_2,r_1) \}\\
    &= \max \left \{\begin{array}{l}
      \max \{\#Y_1, \#Y_2 \} -\#Y_1 + \sum_{\sigma\owns \hat{x}_2} \#\{y\in Y_1^{\sigma}\colon g(y)\notin Y_2^{\sigma} \} \\
      \max \{\#Y_1, \#Y_2 \} -\#Y_2 + \sum_{\sigma\owns \hat{x}_1} \#\{y\in Y_2^{\sigma}\colon g'(y)\notin Y_1^{\sigma} \}
    \end{array} \right\}\\
    &= \max \{\#Y_1, \#Y_2 \} + \max \left \{\begin{array}{l}
      (-\#Y_1) + \sum_{\sigma\owns \hat{x}_2} \#\{y\in Y_1^{\sigma}\colon g(y)\notin Y_2^{\sigma} \} \\
      (-\#Y_2) + \sum_{\sigma\owns \hat{x}_1} \#\{y\in Y_2^{\sigma}\colon g'(y)\notin Y_1^{\sigma} \}
    \end{array} \right\}\\
    &= \max \{\#Y_1, \#Y_2 \} - \min \left \{\begin{array}{l}
      \#Y_1 - \sum_{\sigma\owns \hat{x}_2} \#\{y\in Y_1^{\sigma}\colon g(y)\notin Y_2^{\sigma} \} \\
      \#Y_2 - \sum_{\sigma\owns \hat{x}_1} \#\{y\in Y_2^{\sigma}\colon g'(y)\notin Y_1^{\sigma} \}
    \end{array} \right\}\\
    &\le \max \{\#Y_1, \#Y_2 \} -  \min \left \{\begin{array}{l}
      \#Y_1 - \#Y_1^{\not\pis} + \kappa(r_1,r_2) \\
      \#Y_2 - \#Y_2^{\not\pis} + \kappa(r_2,r_1)
    \end{array} \right\}\\
    &= \max \{\#Y_1,\#Y_2\} - \min \left\{ \#Y_1^{\pis} + \kappa(r_1,r_2), \#Y_2^{\pis} + \kappa(r_2,r_1) \right\}. \qedhere
  \end{align*}
\end{proof}

\subsection{General Statement; Interpretation}
\label{sec:general_statement}

Combining Lemmas \ref{lemma:minimumweight} and \ref{lemma:distance}, we can state the general result:

\begin{theorem}
  \label{theorem:bounds}
  The minimum function weight and the distance between two $\relX$ objects $r_1, r_2$ are given by
  \begin{align*}
    \min_{g} w(g|r_1,r_2) &= \max \{\#Y_1, \#Y_2\} - \#Y_1 + \sum_{\sigma\owns \hat{x}} \#\{y\in Y_1^{\sigma}\colon g(y)\notin Y_2^{\sigma}\}\\ &\le \max \{\#Y_1, \#Y_2\} - (\#Y_1^{\pis} + \kappa(r_1,r_2))
  \end{align*}
  and
  \begin{align*} d(r_1, r_2) &= \max \{\#Y_1, \#Y_2 \} - \min \left
    \{\begin{array}{l}
      \#Y_1 - \sum_{\sigma\owns \hat{x}_2} \#\{y\in Y_1^{\sigma}\colon g(y)\notin Y_2^{\sigma} \} \\
      \#_2 - \sum_{\sigma\owns \hat{x}_1} \#\{y\in Y_2^{\sigma}\colon g'(y)\notin Y_1^{\sigma} \}
    \end{array} \right\}\\
    &\le \max \{\#Y_1,\#Y_2\} - \min \left\{ \#Y_1^{\pis} + \kappa(r_1,r_2), \#Y_2^{\pis} + \kappa(r_2,r_1) \right\}
  \end{align*}
\end{theorem}

When there are morphisms in both directions, the summations in the exact forms of weight and distance vanish, and the distance formula resolves to
\begin{equation*}
  d(r_1,r_2) = \max \{\#Y_1, \#Y_2\} - \min \{\#Y_1, \#Y_2\} = |\#Y_1 - \#Y_2|,
\end{equation*}
consistent with the intuition that the only difference between the $\relX$ objects is the number of elements in the $Y$ sets. Similarly, when a morphism is available only in the direction $r_1\to r_2$, distance resolves to
\begin{align*}
  d(r_1, r_2) &= \max \{\#Y_1, \#Y_2\} - \min \left \{ \#Y_1, \#Y_2 -  \sum_{\sigma\owns \hat{x}_1} \#\{y\in Y_2^{\sigma}\colon g(y)\notin Y_1^{\sigma} \} \right\}\\
  &\le \max \{\#Y_1, \#Y_2\} - \min \left \{ \#Y_1, \#Y_2^{\pis} + \kappa(r_2,r_1) \right\}.
\end{align*}

Comparing the distances when there are morphisms in both directions and when there is one in only one direction, we are led to interpret the quantity $\kappa(r_1, r_2)$ as capping the weight ``penalty'' that non-morphism mappings in a particular direction impose compared to morphisms. Overall, the weight and distance between two arbitrary $\relX$ objects depend on whether and in which direction(s) morphisms exist between them and in each case are bounded above by quantities that depend solely on the relations embedded in the $\relX$ objects and that act to bound the impact of the absence of morphisms.

\begin{example}
  \label{eg:eg_kappa_nomorphs}
  Consider two $\relX$ objects $r_1, r_1$ defined over the same $X$
  and $Y$ but with different relations \begin{align*}
  R_1 & = \begin{pmatrix}
    1 & 0 & 1 & 1 & 0 & 0 & 1 & 1 & 0 & 1\\
    0 & 1 & 0 & 1 & 0 & 0 & 0 & 1 & 1 & 1\\
    1 & 1 & 1 & 0 & 1 & 0 & 0 & 0 & 0 & 1\\
    1 & 1 & 1 & 0 & 1 & 0 & 0 & 0 & 0 & 0\\
    1 & 1 & 1 & 1 & 1 & 0 & 1 & 1 & 1 & 1
  \end{pmatrix}\\
  R_2 & = \begin{pmatrix}
    0 & 0 & 1 & 1 & 0 & 1 & 1 & 1 & 0 & 1\\
    0 & 0 & 1 & 1 & 1 & 1 & 0 & 1 & 1 & 0\\
    1 & 0 & 1 & 1 & 0 & 1 & 1 & 0 & 1 & 1\\
    0 & 1 & 0 & 1 & 0 & 0 & 0 & 1 & 1 & 0\\
    0 & 0 & 0 & 0 & 0 & 1 & 0 & 0 & 1 & 1
  \end{pmatrix}
  \end{align*}
  Sorting their columns
  \begin{align*}
    R_1' &= \left( \begin{matrix}
      0\\
      0\\
      0\\
      0\\
      0
    \end{matrix}\medspace\medspace \rule[-6ex]{1pt}{13.5ex}\medspace\medspace
    \begin{matrix}
      1 & 0 & 1 & 1 & 1 & 0 & 1 & 1 & 0\\
      0 & 1 & 1 & 1 & 1 & 0 & 0 & 0 & 1\\
      0 & 0 & 0 & 0 & 1 & 1 & 1 & 1 & 1\\
      0 & 0 & 0 & 0 & 0 & 1 & 1 & 1 & 1\\
      1 & 1 & 1 & 1 & 1 & 1 & 1 & 1 & 1
    \end{matrix} \right) \\
    &\phantom{=\quad\
    \begin{matrix}
      0 & 1 & 1 & 1 & 1 & 0 & 0 & 0 & 1
    \end{matrix}}\ \thinspace \Updownarrow\\
    R_2' &= \begin{pmatrix}
      0 & 0 & 1 & 1 & 0 & 1 & 1 & 1 & 1 & 0\\
      1 & 0 & 0 & 1 & 0 & 1 & 1 & 0 & 1 & 1\\
      0 & 1 & 1 & 1 & 0 & 0 & 1 & 1 & 1 & 1\\
      0 & 0 & 0 & 0 & 1 & 1 & 1 & 0 & 0 & 1\\
      0 & 0 & 0 & 0 & 0 & 0 & 0 & 1 & 1 & 1
    \end{pmatrix}
  \end{align*}
  we find agreement in only one column (marked by $\Updownarrow$). Exact calculation of distance between the two requires checking all $f\colon Y\to Y$ excluding the one agreeing $y$ in light of both $R_1$ and $R_2$, of which there are $2\times 9^{9} \approx 775$ million. Applying the algorithm, however, inspection of $R_1'$ and $R_2'$ readily demonstrates that we can sort $R_1'$ into two disjoint partition blocks (separated above by a vertical bar) but $R_2'$ cannot be sorted into disjoint partition blocks. Consequently $\kappa(r_1,r_2)=1$ and $\kappa(r_1,r_2)=0$, giving the distance between the two relations
  \begin{equation*}
    d(r_1, r_2) \le \max\{10,10\}-\min\{1 + 0,1 + 1\} = 10 - 0 = 9.
  \end{equation*}
\end{example}

\begin{example}
  \label{eg:eg_kappamorph-nomorph}
  Now consider two $\relX$ objects $r_1, r_2$ defined over the same $X$ and $Y$ but with the following relations that differ only slightly (one change in each of columns 4, 7, 8, and 10)
  \begin{align*}
    R_1 & = \begin{pmatrix}
      0 & 0 & 1 & 0 & 0 & 1 & 0 & 0 & 0 & 1\\
      0 & 0 & 1 & 0 & 1 & 0 & 0 & 0 & 0 & 1\\
      0 & 1 & 0 & 1 & 0 & 0 & 1 & 0 & 1 & 0\\
      0 & 0 & 0 & 0 & 0 & 0 & 0 & 0 & 0 & 0\\
      0 & 1 & 0 & 0 & 0 & 0 & 1 & 0 & 0 & 0
    \end{pmatrix}\\
    R_2 & = \begin{pmatrix}
      0 & 0 & 1 & 0 & 0 & 1 & 0 & 0 & 0 & 1\\
      0 & 0 & 1 & 1 & 1 & 0 & 0 & 0 & 0 & 0\\
      0 & 1 & 0 & 1 & 0 & 0 & 0 & 0 & 1 & 0\\
      0 & 0 & 0 & 0 & 0 & 0 & 0 & 0 & 0 & 1\\
      0 & 1 & 0 & 0 & 0 & 0 & 1 & 1 & 0 & 0\\
    \end{pmatrix}
  \end{align*}
  Sorting them into
  \begin{align*}
    R_1' &= \left( \begin{matrix}
      0 & 0\\
      0 & 0\\
      0 & 0\\
      0 & 0\\
      0 & 0
    \end{matrix} \medspace\medspace \rule[-6ex]{1pt}{13.5ex}\medspace\medspace
    \begin{matrix}
      1 & 0 & 1 & 1\\
      0 & 1 & 1 & 1\\
      0 & 0 & 0 & 0\\
      0 & 0 & 0 & 0\\
      0 & 0 & 0 & 0
    \end{matrix} \medspace\medspace \rule[-6ex]{1pt}{13.5ex}\medspace\medspace
    \begin{matrix}
      0 & 0 & 0 & 0\\
      0 & 0 & 0 & 0\\
      1 & 1 & 1 & 1\\
      0 & 0 & 0 & 0\\
      0 & 0 & 1 & 1
    \end{matrix} \right) \\
    R_2' &= \left( \begin{matrix}
      0\\
      0\\
      0\\
      0\\
      0
    \end{matrix} \medspace\medspace \rule[-6ex]{1pt}{13.5ex}\medspace\medspace
    \begin{matrix}
      1 & 0 & 1 & 0 & 0 & 1 & 0 & 0 & 0\\
      0 & 1 & 1 & 0 & 1 & 0 & 0 & 0 & 0\\
      0 & 0 & 0 & 1 & 1 & 0 & 0 & 0 & 1\\
      0 & 0 & 0 & 0 & 0 & 1 & 0 & 0 & 0\\
      0 & 0 & 0 & 0 & 0 & 0 & 1 & 1 & 1
    \end{matrix} \right)
  \end{align*}
  and removing one-for-one matching columns we find the following column differences with partitions
  \begin{align*}
    R_1' - R_2' &= \left( \begin{matrix}
      0\\
      0\\
      0\\
      0\\
      0
    \end{matrix} \medspace\medspace \rule[-6ex]{1pt}{13.5ex}\medspace\medspace
    \begin{matrix}
      1\\
      1\\
      0\\
      0\\
      0
    \end{matrix} \medspace\medspace \rule[-6ex]{1pt}{13.5ex}\medspace\medspace
    \begin{matrix}
      0 & 0\\
      0 & 0\\
      1 & 1\\
      0 & 0\\
      0 & 1
    \end{matrix} \right)\\
    R_2' - R_1' &= \left( \begin{matrix}
      0\\
      1\\
      1\\
      0\\
      0
    \end{matrix} \medspace\medspace \rule[-6ex]{1pt}{13.5ex}\medspace\medspace
    \begin{matrix}
      1\\
      0\\
      0\\
      1\\
      0
    \end{matrix} \medspace\medspace \rule[-6ex]{1pt}{13.5ex}\medspace\medspace
    \begin{matrix}
      0 & 0\\
      0 & 0\\
      0 & 0\\
      0 & 0\\
      1 & 1
    \end{matrix} \right).
  \end{align*}
  Exact calculation of distance requires considering $2\times 4^{4} \approx 512$ possible mappings. Applying the $\kappa$ algorithm, however, we find the disagreeing columns of both relations can each be partitioned into 2 blocks of 1 column and one block of 2 columns. Since we map 4 disagreeing columns in each direction, we find $\kappa(r_1,r_2) = \kappa(r_2,r_1) = 2$ and an upper bound on the distance between the two relations is
  \begin{equation*}
    d(r_1, r_2) \le \max \{10, 10\} - \min \{6 + 2, 6 + 2 \} = 2.
  \end{equation*}
\end{example}

\section{Conclusion}
\label{sec:conclusion}

Our exploration has demonstrated that the relation multiset $M(X,Y,R)$ is a full and faithful functor from the $\rel$ category of relations to the $\mul$ category of multisets. Given our novel definition of a pseudometric for relations, applying the functor to the subcategory $\relX$ of relations sharing a common feature set $X$ facilitated establishing the $\kappa$ bound on the pseudometric for $\relX$ objects. We also specified an algorithm for calculating that bound that avoids a potentially expensive search of the combinatorial space of potential mappings between relations. The $\kappa$ bound and algorithm can be used to bound differences between binary data sets about common features ($X$), for instance in the search for consensus specifications.

This paper has focused exclusively on simple relations $R\colon X\times Y \to \{0,1\}$, corresponding to binary data about the presence or absence of features. Future work could extend the results to data with partially ordered values about features, $R\colon X\times Y \to P$, where $P$ has a partial order, which is isomorphic to the simple case (\cite[]{robinson_cosheaf_2020}, pp.~12-14).  Another possible extension is to apply the algebra of multisets to calculate distances and bounds on subsets or filtrations of relations, for instance, to distinguish locally similar or different subsets or to reduce computation and identify ``new'' information as data sets accrete or expand over time due to repeated sampling, changed samplers, expanded feature sets, etc.

\section*{Acknowledgments}
This material is based upon work supported by the Defense Advanced Research Projects Agency (DARPA) SafeDocs program under contract HR001119C0072.  Any opinions, findings and conclusions or recommendations expressed in this material are those of the authors and do not necessarily reflect the views of DARPA.  The authors would like to thank the SafeDocs test and evaluation team, including NASA (National Aeronautics and Space Administration) Jet Propulsion Laboratory, California Institute of Technology and the PDF Association, Inc., for providing the test data.

\section*{Conflict of interest}
The authors state that there is no conflict of interest.

\bibliographystyle{plainnat}
\bibliography{relations2_bib}
\end{document}